\documentclass{lmcs}
\pdfoutput=1
\usepackage[utf8]{inputenc}

\usepackage{lastpage}
\lmcsdoi{21}{2}{21}
\lmcsheading{}{\pageref{LastPage}}{}{}%
{May~16,~2024}{Jun.~09,~2025}{}

\keywords{matrix semigroup, special affine group, Identity Problem, Group Problem}

\usepackage{textalpha}
\usepackage{mathtools}
\usepackage{array}
\usepackage[ruled]{algorithm}
\usepackage{xcolor}
\usepackage{amsthm}
\usepackage{amsmath}
\usepackage{amssymb}
\usepackage{amsfonts}
\usepackage{graphicx}
\usepackage{thm-restate}
\usepackage{url}
\usepackage{mathtools}
\usepackage{enumitem}
\usepackage{hyperref}
\theoremstyle{plain} 
\newcounter{ProblemCounter}
\allowdisplaybreaks

\newcommand{\Z}{\mathbb{Z}}

\newcommand{\Q}{\mathbb{Q}}
\newcommand{\R}{\mathbb{R}}
\newcommand{\C}{\mathbb{C}}
\newcommand{\K}{\mathbb{K}}

\newcommand{\Rp}{\mathbb{R}_{\geq 0}}
\newcommand{\Rpp}{\mathbb{R}_{>0}}
\newcommand{\Qp}{\mathbb{Q}_{\geq 0}}
\newcommand{\Zp}{\mathbb{Z}_{\geq 0}}
\newcommand{\Zpp}{\mathbb{Z}_{>0}}

\newcommand{\SL}{\mathsf{SL}}

\newcommand{\T}{\mathsf{T}}

\newcommand{\SA}{\mathsf{SA}(2, \mathbb{Z})}

\newcommand{\HH}{\operatorname{H}}

\newcommand{\Lat}{\operatorname{Lat}}

\newcommand{\diag}{\operatorname{diag}}

\newcommand{\Cell}{\operatorname{Cell}}
\newcommand{\mG}{\mathcal{G}}
\newcommand{\sgmG}{\langle \mathcal{G} \rangle}

\newcommand{\mR}{\mathcal{R}}
\newcommand{\mL}{\mathcal{L}}

\newcommand{\mS}{\mathcal{S}}
\newcommand{\mH}{\mathcal{H}}

\newcommand{\mA}{\mathcal{A}}

\newcommand{\bv}{\boldsymbol{v}}
\newcommand{\bw}{\boldsymbol{w}}
\newcommand{\bx}{\boldsymbol{x}}
\newcommand{\byy}{\boldsymbol{y}}

\newcommand{\ba}{\boldsymbol{a}}
\newcommand{\bb}{\boldsymbol{b}}
\newcommand{\bc}{\boldsymbol{c}}
\newcommand{\bd}{\boldsymbol{d}}

\newcommand{\bn}{\boldsymbol{n}}
\newcommand{\be}{\boldsymbol{e}}

\newcommand{\bt}{\boldsymbol{t}}
\newcommand{\bs}{\boldsymbol{s}}
\newcommand{\bl}{\boldsymbol{\ell}}
\newcommand{\bzer}{\boldsymbol{0}}

\begin{document}

\title[The Identity Problem in the special affine group of $\mathbb{Z}^2$]{The Identity Problem in the special affine group of $\mathbb{Z}^2$}
\thanks{The author acknowledges support from UKRI Frontier Research Grant EP/X033813/1.}	

\author[R.~Dong]{Ruiwen Dong\lmcsorcid{0009-0007-4349-082X}}[a, b]

\address{Department of Mathematics, Saarland University, Germany}	

\address{Magdalen College, University of Oxford, United Kingdom}	
\email{ruiwen.dong@magd.ox.ac.uk}  

\begin{abstract}
\noindent We consider semigroup algorithmic problems in the Special Affine group $\SA = \Z^2 \rtimes \SL(2, \Z)$,
which is the group of affine transformations of the lattice $\Z^2$ that preserve orientation.
Our paper focuses on two decision problems introduced by Choffrut and Karhum\"{a}ki (2005): the \emph{Identity Problem} (does a semigroup contain a neutral element?) and the \emph{Group Problem} (is a semigroup a group?) for finitely generated sub-semigroups of $\SA$.
We show that both problems are decidable and \textbf{NP}-complete.
Since $\SL(2, \Z) \leq \SA \leq \SL(3, \Z)$, our result extends that of Bell, Hirvensalo and Potapov (2017) on the \textbf{NP}-completeness of both problems in $\SL(2, \Z)$,
and contributes a first step towards the open problems in $\SL(3, \Z)$.
\end{abstract}

\maketitle

\section{Introduction}
\subsection*{Algorithmic problems in matrix semigroups}
The computational theory of matrix groups and semigroups is one of the
oldest and most well-developed parts of computational algebra.  
The area plays an essential role in analysing system dynamics, and has numerous applications in automata theory, program analysis, and interactive proof systems~\cite{beals1993vegas, blondel2005decidable, choffrut2005some, derksen2005quantum, hrushovski2018polynomial}.
The earliest studied algorithmic problems for groups and semigroups are the \emph{Semigroup Membership} and the \emph{Group Membership} problems, introduced respectively by Markov~\cite{markov1947certain} and Mikhailova~\cite{mikhailova1966occurrence} in the 1940s and 1960s.
For both problems, we work in a fixed matrix group $G$.
The input is a finite set of matrices $\mG = \{A_1, \ldots, A_K\}$ in $G$ and a target matrix $A \in G$.
Denote by $\sgmG$ the semigroup generated by $\mG$, and by $\langle\mG\rangle_{grp}$ the group generated by $\mG$.
\begin{enumerate}[label = \roman*.]
    \item \textit{(Semigroup Membership)} decide whether $\sgmG$ contains $A$.
    \item \textit{(Group Membership)} decide whether $\langle\mG\rangle_{grp}$ contains $A$.
    \setcounter{ProblemCounter}{\value{enumi}}
\end{enumerate}
Both problems are undecidable in general matrix groups by classical results of Markov and Mikhailova~\cite{markov1947certain, mikhailova1966occurrence}.
In this paper, we consider the \emph{Identity Problem} and the \emph{Group Problem}, introduced by Choffrut and Karhum\"{a}ki~\cite{choffrut2005some} in 2005.
These two decision problems concern the \emph{structure} of semigroups rather than their \emph{membership}.
Fix a matrix group $G$, the input for both problems is a finite set of matrices $\mG = \{A_1, \ldots, A_K\}$ in $G$.
\begin{enumerate}[label = \roman*.]
    \setcounter{enumi}{\value{ProblemCounter}}
    \item \textit{(Identity Problem)} decide whether $\sgmG$ contains the neutral element $I$ of $G$.
    \item \textit{(Group Problem)} decide whether $\langle\mG\rangle$ is a group.
    \setcounter{ProblemCounter}{\value{enumi}}
\end{enumerate}
Apart from their obvious use in determining structural properties of a semigroup\footnote{For example, given a decision procedure for the Group Problem, one can compute a generating set for the \emph{group of units}  (set of of invertible elements) of a finitely generated semigroup $\sgmG$~\cite{dong2022identity}.}, the Identity Problem and the Group Problem are closely related to\footnote{In fact, decidability of Semigroup Membership subsumes decidability of the Identity Problem and the Group Problem.} the more difficult Semigroup Membership problem.
Usually, the solution to the Identity Problem is the most essential special
case on the way to building an algorithm for
the Semigroup Membership problem.
We also point out that there are significantly more available algorithms for groups than there are for semigroups, therefore performing preliminary checks using the Group Problem can help decide Semigroup Membership in many special cases (see Section~\ref{sec:obs}).

All four algorithmic problems remain undecidable even for matrices of small dimensions.
The \emph{Special Linear group} of dimension $n$, denoted by $\SL(n, \Z)$, is defined as the group of $n \times n$ integer matrices with determinant one.
Mikhailova famously showed undecidability of the Group Membership problem (and hence also Semigroup Membership) in $\SL(4, \Z)$~\cite{mikhailova1966occurrence}.
Later, Bell and Potapov showed undecidability of the Identity Problem and the Group Problem in $\SL(4, \Z)$~\cite{bell2010undecidability}.
Both undecidability results stem from the fact that $\SL(4, \Z)$ contains as a subgroup a direct product of two non-abelian free groups.
In dimension two, the Semigroup Membership problem in $\SL(2, \Z)$ was shown to be decidable in \textbf{EXPSPACE} by Choffrut and Karhum\"{a}ki~\cite{choffrut2005some}; Group Membership was shown to be in \textbf{PTIME} by Lohrey~\cite{lohrey2021subgroup}; the Identity Problem and the Group Problem were shown to be \textbf{NP}-complete by Bell, Hirvensalo and Potapov~\cite{bell2017identity}.
Here, elements of $\SL(2, \Z)$ are represented using matrices with binary encoded entries.
It remains an intricate open problem whether any of these four algorithmic problems is decidable in $\SL(3, \Z)$.
Nevertheless, Ko, Niskanen and Potapov~\cite{ko2017identity} recently showed that $\SL(3, \Z)$ cannot embed pairs of words over an alphabet of size two, suggesting that all four problems in $\SL(3, \Z)$ might be decidable.
The following table~\ref{tbl:stateofart} summarizes the state of art as well as our result, the group $\SA$ will be introduced in the next subsection.

\begin{table}[!ht]
\begin{center}
\begin{tabular}{ | m{1.4cm} | m{4cm}| m{3.5cm} | m{4.5cm} | } \hline
   & Semigroup Membership & Group Membership & Identity \& Group Problem \\
  \hline
  \hline
  $\SL(2, \Z)$ & EXPSPACE~\cite{choffrut2005some} & PTIME~\cite{lohrey2021subgroup} & NP-complete~\cite{bell2017identity}\\ 
  \hline
  $\SA$ & ? & Decidable~\cite{delgado2017extensions} & NP-complete$^{*}$\\   
  \hline
  $\SL(3, \Z)$ & ? & ? & ? \\ 
  \hline
  $\SL(4, \Z)$ & Undec.~\cite{mikhailova1966occurrence} & Undec.~\cite{mikhailova1966occurrence} & Undec.~\cite{bell2010undecidability} \\ 
  \hline
\end{tabular}
\end{center}
\caption{\label{tbl:stateofart} * = Our result.}
\end{table}

\subsection*{The Special Affine group $\SA$}
In this paper we focus on an intermediate group between $\SL(2, \Z)$ and $\SL(3, \Z)$: the \emph{Special Affine group} $\SA$.
This affine analogue of $\SL(2, \Z)$ is defined as the group of affine transformations of $\Z^2$ that preserve orientation.
Written as matrices, elements of $\SA$ are $3 \times 3$ integer matrices of the following form.
\[
\SA \coloneqq \left\{ 
\begin{pmatrix}
        A & \ba \\
        0 & 1 \\
\end{pmatrix}
\;\middle|\;
A \in \SL(2, \Z), \ba \in \Z^2
\right\}.
\]
To be precise, elements of $\SA$ are represented using matrices with \emph{binary encoded} entries.
We denote by $(A, \ba)$ the element 
$\begin{pmatrix}
        A & \ba \\
        0 & 1 \\
\end{pmatrix}$, then the neutral element of $\SA$ is $(I, \bzer)$; multiplication in $\SA$ is given by
\[
(A, \ba) \cdot (B, \bb) = (AB, A \bb + \ba).
\]
Naturally, $\SA$ has a subgroup $\{(A, \bzer) \mid A \in \SL(2, \Z)\}$ $\cong \SL(2, \Z)$.

Special Affine groups are important in the context of many fundamental problems, such as Lie groups~\cite{wolf1963affine}, polyhedral geometry~\cite{mundici2014invariant}, dynamical systems~\cite{cabrer2017classifying}, quadrics~\cite{i2011affine}, computer vision~\cite{giefer2020extended, kwon2010visual} and gauge theory~\cite{alday2010affine}.
Apart from the intrinsic interest to study $\SA$, we also point out that the Special Affine group has tight connections to various reachability problems.
Some of the central questions in automated verification include reachability problems in \emph{Affine Vector Addition Systems} and \emph{Affine Vector Addition Systems with states (Affine VASS)} over the integers~\cite{raskin2021affine}.
While both problems as well as many of their variations have been shown to be decidable in dimension one and undecidable for dimension three~\cite{finkel2013reachability, ko2018reachability}, few results are known for dimension two.
Since the study of these reachability problems in dimension two necessitates the study of sub-semigroups of $\SA$, the techniques introduced in our paper may provide insights into these open problems.

Currently, among the four algorithmic problems introduced in the beginning, the only known result in $\SA$ is decidability of the Group Membership problem.
This can be deduced from the recent work of Delgado~\cite{delgado2017extensions}, who showed decidability of the Group Membership problem in the semidirect product $\Z^m \rtimes F$, where $F$ is a free group.
Delgado's result relies on a generalization of the Stalling automata, and can therefore be extended to the case where $F$ is \emph{virtually free}.
This can then be applied to $\SA = \Z^2 \rtimes \SL(2, \Z)$, since $\SL(2, \Z)$ is virtually free.

In this paper, we step further by considering sub-semigroups of $\SA$.
We show decidability and \textbf{NP}-completeness of the Identity Problem and the Group Problem in $\SA$.
This extends the \textbf{NP}-completeness result of Bell et al.\ for the Identity Problem and the Group Problem in $\SL(2, \Z)$, and contributes a first step towards solving problems in $\SL(3, \Z)$.

The \textbf{NP}-hard lower bounds in $\SA$ directly follows from its embedding of the subgroup $\SL(2, \Z)$.
In order to prove decidability and the $\textbf{NP}$ upper bounds, our main idea is the following.
Given a finitely generated sub-semigroup $\sgmG$ of $\SA$, we show that we can without loss of generality suppose its image under the natural projection $p \colon \SA \rightarrow \SL(2, \Z)$ to be a group, using a classic result for $\SL(2, \Z)$.
If $p(\sgmG)$ is a group, we then invoke an effective version of the \emph{Tits alternative}, which shows that either $p(\sgmG)$ contains a non-abelian free subgroup, or it is virtually solvable.
In the first case we show that $\sgmG$ is always a group, while in the second case we further simplify the problem by identifying six subcases for $p(\sgmG)$.
Our proof combines two viewpoints: a group theoretic viewpoint of $\SL(2, \Z)$ as a virtually free group, and a geometric viewpoint of $\SA$ as elements acting on the lattice $\Z^2$.

Beyond the Identity Problem and the Group Problem, we will discuss some obstacles to generalizing our results to the Semigroup Membership problem in $\SA$.
Our results actually show that Semigroup Membership in $\SA$ is decidable in many cases under additional constraints.
We identify one of the remaining difficult cases,
namely when $\sgmG$ is isomorphic to a sub-semigroup of the semidirect product $\Z[\lambda] \rtimes_{\lambda} \Z$, where $\lambda$ is a quadratic integer.
The Semigroup Membership problem in $\Z[\lambda] \rtimes_{\lambda} \Z$ remains open.
However, the group $\Z[\lambda] \rtimes_{\lambda} \Z$ bears certain similarities to the \emph{Baumslag-Solitar group} $\mathsf{BS}(1, q) \coloneqq \Z[\frac{1}{q}] \rtimes_{q} \Z$;
and a recent result by Cadilhac, Chistikov and Zetzsche~\cite{DBLP:conf/icalp/CadilhacCZ20} showed decidability of the \emph{rational subset membership problem}\footnote{The rational subset membership problem subsumes the Semigroup Membership problem.} in $\mathsf{BS}(1, q)$ by considering rational languages of \emph{base-$q$ expansions}.
Despite some visible difficulties, it would be interesting in the future to adapt this approach to study the Semigroup Membership problem in $\Z[\lambda] \rtimes_{\lambda} \Z$, namely by considering rational languages of \emph{base-$\lambda$ expansions}~\cite{BLANCHARD1989131}, where $\lambda$ is an algebraic integer.

\section{Preliminaries}\label{sec:prelim}
\subsection*{Words, semigroups and groups}
Let $G$ be an arbitrary group.
Let $\mA = \{a_1, \ldots, a_K\}$ be a set of elements in $G$.
Considering $\mA$ as an alphabet, denote by $\mA^*$ the set of words over $\mA$.
For an arbitrary word $w = a_{i_1} a_{i_2} \cdots a_{i_m} \in \mA^*$, by multiplying consecutively the elements appearing in $w$, we can evaluate $w$ as an element $\pi(w)$ in $G$.
We say that the word $w$ \emph{represents} the element $\pi(w)$.
The semigroup $\langle \mA \rangle$ generated by $\mA$ is hence the set of elements in $G$ that are represented by \emph{non-empty} words in $\mA^*$.

A word $w$ over the alphabet $\mA$ is called \emph{full-image} if every letter in $\mA$ has at least one occurrence in $w$.

\begin{restatable}{lem}{lemgrpword}\label{lem:grpword}
    Let $\mA = \{a_1, \ldots, a_K\}$ be a set of elements in a group $G$.
    Consider the following conditions:
    \begin{enumerate}[label = (\roman*)]
        \item The neutral element $I$ of $G$ is represented by a full-image word over $\mA$.
        \item The semigroup $\langle \mA \rangle$ is a group.
        \item Every element $A \in \langle \mA \rangle$ is represented by a full-image word over $\mA$.
    \end{enumerate}
    Then $(i) \Longleftrightarrow (ii)$, and $(ii) \Longrightarrow (iii)$.
\end{restatable}
\begin{proof}
    $(i) \Longrightarrow (ii)$.
    Let $w \in \mA^*$ be a full-image word with $\pi(w) = I$.
    Then for every $i$, the word $w$ can be written as $w = v a_i v'$, so $a_i^{-1} = \pi(v') \pi(v) \in \langle \mA \rangle$.
    Therefore, the semigroup $\langle \mA \rangle$ contains all the inverse $a_i^{-1}$, and is thus a group.
    
    $(ii) \Longrightarrow (iii)$.
    If $\langle \mA \rangle$ is a group, then for all $i$, the inverse $a_i^{-1}$ can be written as $\pi(w_i)$ for some word $w_i \in \mA^*$.
    Then for any element $A \in \langle \mA \rangle$, represented by some word $w_A \in \mA^*$, the word $w \coloneqq w_A a_1 w_1 a_2 w_2 \cdots a_K w_K$ is a full-image word with $\pi(w) = \pi(w_A) \pi(a_1 w_1) \cdots \pi(a_K w_K) = A$.
    
    $(ii) \Longrightarrow (i)$.
    If $\langle \mA \rangle$ is a group, then $I \in \langle \mA \rangle$, so by $(ii) \Longrightarrow (iii)$, $I$ can be represented by a full-image word.
\end{proof}

Let $\mG = \{(A_1, \ba_1), \ldots, (A_K, \ba_K)\}$ be a set of elements in $\SA$.
Suppose that an element $A \in \langle A_1, \ldots, A_K \rangle$ in $\SL(2, \Z)$ is represented by a full-image word $w$ in the alphabet $\{A_1, \ldots, A_K\}$.
Then replacing each letter $A_i$ in $w$ by $(A_i, \ba_i)$, we obtain a product $(A, \ba) \in \SA$ for some $\ba \in \Z^2$, represented by a full-image word over $\mG$.

The following observation shows that decidability of the Group Problem implies decidability of the Identity Problem.

\begin{lemC}[\cite{bell2010undecidability}]\label{lem:grptoid}
    Given a subset $\mA$ of a group $G$, the semigroup $\langle \mA \rangle$ contains the neutral element $I$ if and only if there exists a non-empty subset $\mH \subseteq \mA$ such that $\langle \mH \rangle$ is a group.
\end{lemC}

\subsection*{Linear algebra}
For an arbitrary matrix $A \in \SL(2, \Z)$, an \emph{invariant subspace} of $A$ is a $\C$-linear subspace $V$ of $\C^2$ such that $A V = V$. 
If the eigenvalues of $A$ are reals, then one can suppose that the invariant spaces of $A$ are subspaces of $\R^2$.
Denote by $\Lat(A)$ the set of \emph{dimension one} invariant subspaces of $A$.
It is easy to see that if $A \notin \{I, -I\}$, then all the eigenspaces of $A$ have dimension one, so $\Lat(A)$ has one or two elements.

Two matrices $A$ and $B$ are called \emph{conjugates} over a field $\K$ if $P^{-1} A P = B$ for some matrix $P$ with entries in $\K$.
This is denoted as $A \overset{\K}{\sim} B$.
Two matrices $A, B \in \SL(2, \Z)$ are called \emph{simultaneously triangularizable} if there exists a complex matrix $P$ such that $P^{-1} A P$ and $P^{-1} B P$ are both upper-triangular.
It is easy to see that if $\Lat(A) \cap \Lat(B) \neq \emptyset$, then $A$ and $B$ are simultaneously triangularizable.

\subsection*{Group theory}
For a general reference on group theory, see~\cite{dructu2018geometric}.
\begin{defi}\label{def:solv}
    A group $G$ is called \emph{solvable} if its \emph{derived series}
    \[
    G = G^{(0)} \trianglerighteq G^{(1)} \trianglerighteq G^{(2)} \trianglerighteq \cdots,
    \]
    where $G^{(i+1)}$ is the commutator subgroup\footnote{The commutator subgroup of a group $H$ is the subgroup generated by the elements $ghg^{-1}h^{-1}$ where $g, h \in H$.} of $G^{(i)}$, eventually reaches the trivial group.
\end{defi}
Every subgroup of a solvable group is solvable~\cite[Proposition~13.91]{dructu2018geometric}.
Abelian groups are obviously solvable.
For any field $\K$ and integer $n$, denote by $\T(n, \K)$ the group of $n \times n$ invertible upper-triangular matrices with entries in $\K$.
Then $\T(n, \K)$ is a solvable group~\cite{beals1999algorithms}.
In particular, if two matrices $A, B \in \SL(2, \Z)$ are simultaneously triangularizable, then the group $G$ they generate is isomorphic to a subgroup of $\T(2, \C)$; thus $G$ is solvable.

Given an alphabet $\Sigma$, define the corresponding group alphabet $\Sigma^{\pm} \coloneqq \Sigma \cup \{a^{-1} \mid a \in \Sigma\}$, where $a^{-1}$ is a new symbol.
There is a natural involution $(\cdot)^{-1}$ over $\left(\Sigma^{\pm}\right)^*$ defined by $(a^{-1})^{-1}$ and $(a_1 \cdots a_m)^{-1} = a_m^{-1} \cdots a_1^{-1}$. 
A word over the alphabet is called \emph{reduced} if it does not contain consecutive letters $a a^{-1}$ or $a^{-1} a$.
For a word $w$ over the alphabet $\Sigma^{\pm}$, define $\operatorname{red}(w)$ to be the reduced word obtained by iteratively replacing consecutive letters $a a^{-1}$ and $a^{-1} a$ with the empty string.
The \emph{free group} $F(\Sigma)$ over $\Sigma$ is then defined as the set of reduced words over the alphabet $\Sigma^{\pm}$, where multiplication is given by $v \cdot w = \operatorname{red}(vw)$, and inversion is given by the involution $(\cdot)^{-1}$.
A group is called \emph{free} if it is a free group over some alphabet.
The free group $F(\Sigma)$ is abelian if and only if the cardinality of $\Sigma$ is zero or one, in which case $F(\Sigma)$ is trivial or isomorphic to the infinite cyclic group $\Z$.

\begin{thmC}[{(Nielsen–Schreier~\cite[Chapter I, Theorem 5]{serre2002trees})}]\label{thm:freesub}
    Every subgroup of a free group is free.
\end{thmC}

Let $G$ be an arbitrary group with neutral element $I$.
An element $T \in G$ is called \emph{torsion} if $T \neq I$ and $T^m = I$ for some $m > 1$.
A group is called \emph{torsion-free} if it does not contain a torsion element.
In particular, a free group is torsion-free.

A group is called \emph{virtually solvable} if it admits a finite index subgroup that is solvable.
Similarly, a group is called \emph{virtually free} if it admits a finite index subgroup that is free.
The following is a classic result.

\begin{thmC}[{\cite{newman1962structure}}]\label{thm:vf}
    The group $\SL(2, \Z)$ is virtually free.
    Moreover, it contains a finite index free subgroup $F(\{S, T\})$ over two generators.
\end{thmC}

Based on this fact, Bell et al.\ showed the following complexity result.

\begin{thmC}[{\cite{bell2017identity}}]\label{thm:SMPSL2}
    The Identity Problem and the Group Problem in $\SL(2, \Z)$ are NP-complete.
\end{thmC}
Here, the input elements in $\SL(2, \Z)$ are represented using matrices with
binary encoded entries.

\subsection*{Classification of elements in $\SL(2, \Z)$}
Let 
$A = 
\begin{pmatrix}
    a & b \\
    c & d \\
\end{pmatrix}
$
be a matrix in $\SL(2, \Z)$.
The characteristic polynomial of $A$ is $f(X) = X^2 - (a + d) X + (ad - bc) = X^2 - (a + d) X + 1$.
Consider the Jordan Normal form of the matrix $A$ in the five following cases.
\begin{enumerate}[label = (\roman*)]
    \item $a + d = 0$.
    
    In this case,
    $
    A \overset{\Q(i)}{\sim}
    \begin{pmatrix}
    i & 0 \\
    0 & -i \\
    \end{pmatrix}
    $,
    the eigenvalues of $A$ are $i$ and $-i$.
    We have $A^4 = I$, so $A$ is a torsion element.
    
    \item $|a + d| = 1$.
    
    In this case, 
    $
    A \overset{\Q(\omega)}{\sim}
    \begin{pmatrix}
    \omega & 0 \\
    0 & \omega^{-1} \\
    \end{pmatrix}
    $,
    where $\omega = \frac{1 + \sqrt{3}i}{2}$ or $\frac{- 1 + \sqrt{3}i}{2}$.
    In both cases, we have $A^6 = I$, so $A$ is a torsion element.
    
    \item $a + d = 2$.
    
    In this case, either $A = I$, or
    $
    A \overset{\Q}{\sim}
    \begin{pmatrix}
    1 & 1 \\
    0 & 1 \\
    \end{pmatrix}
    $.    
    In the second situation, we call $A$ a \emph{shear}.
    The only eigenvalue of $A$ is 1.
    If $A$ is a shear, then $\Lat(A)$ has exactly one element.
    See Figure~\ref{fig:shear} for an illustration.
    
    \item $a + d = - 2$.
    
    In this case, either $A = - I$, or
    $
    A \overset{\Q}{\sim}
    \begin{pmatrix}
    -1 & 1 \\
    0 & -1 \\
    \end{pmatrix}
    $.
    In the second situation, we call $A$ a \emph{twisted inversion}.
    In particular, if $A$ is a twisted inversion, then $(A+I)^2 = 0$, $A^2$ is a shear and $\Lat(A)$ has exactly one element.
    
    \item $|a + d| \geq 3$.
    
    In this case, 
    $
    A \overset{\R}{\sim}
    \begin{pmatrix}
    \lambda & 0 \\
    0 & \lambda^{-1} \\
    \end{pmatrix}
    $,
    where $\lambda$ is the root of $f(X)$ such that $|\lambda| > 1$.
    Furthermore, $\lambda$ is real.
    In this case, we call $A$ a \emph{scale}.
    If $\lambda > 0$, we call $A$ a \emph{positive scale}; if $\lambda < 0$, we call $A$ an \emph{inverting scale}.
    In both cases, $\Lat(A)$ has two elements, one element is the invariant space corresponding to the eigenvalue $\lambda$, and is called the \emph{stretching direction}; the other element is the invariant space corresponding to the eigenvalue $\lambda^{-1}$, and is called the \emph{compressing direction}.
    See Figure~\ref{fig:scale} for an illustration.
\end{enumerate}

    \begin{figure}[ht!]
        \centering
        \begin{minipage}[t]{0.45\textwidth}
            \centering
            \includegraphics[width=\textwidth, keepaspectratio, trim={4cm 1cm 4.5cm 0.5cm},clip]{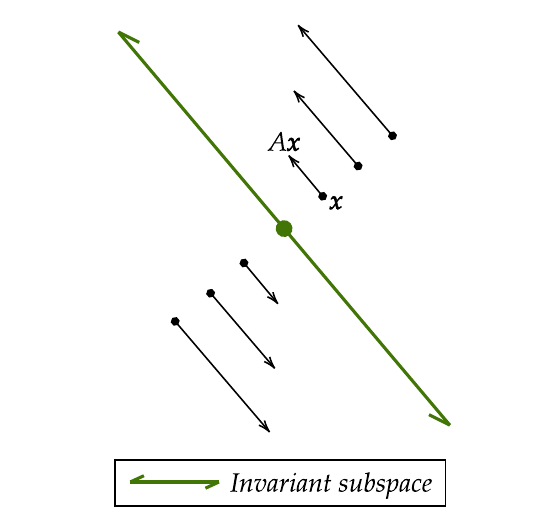}
            \caption{Illustration for a shear.}
            \label{fig:shear}
        \end{minipage}
        \hfill
        \begin{minipage}[t]{.45\textwidth}
            \centering
            \includegraphics[width=\textwidth, keepaspectratio, trim={4.5cm 1.5cm 4cm 0.5cm},clip]{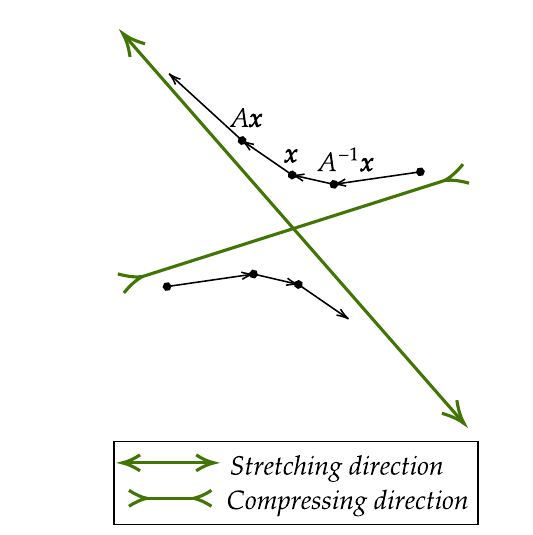}
            \caption{Illustration for a (positive) scale.}
            \label{fig:scale}
        \end{minipage}
    \end{figure}

\section{Overview of decision procedures}\label{sec:overview}
In this section we give an overview of the decision procedures for the Identity Problem and the Group Problem in $\SA$.
We state two propositions (Propositions~\ref{prop:nonsimid} and \ref{prop:simid}) regarding the structure of sub-semigroups of $\SA$.
Assuming these propositions, we prove \textbf{NP}-completeness of the Identity Problem and the Group Problem in $\SA$.
The proofs of Propositions~\ref{prop:nonsimid} and \ref{prop:simid} are delayed until Section~\ref{sec:nonsim} and \ref{sec:sim}.

We now first focus on solving the Group Problem, because by 
Lemma~\ref{lem:grptoid}, decidability of the Group Problem will imply decidability of the Identity Problem.
Fix a set $\mG \coloneqq \{(A_1, \ba_1), \ldots, (A_K, \ba_K)\}$ of elements in $\SA$.
The following lemma shows that we only need to consider the case where $G \coloneqq \langle A_1, \ldots, A_K \rangle$ is a group.

\begin{lem}\label{lem:isgrp}
    Let $\mG \coloneqq \{(A_1, \ba_1), \ldots, (A_K, \ba_K)\}$ be a set of elements of $\SA$.
    If the semigroup $G \coloneqq \langle A_1, \ldots, A_K \rangle$ is not a group, then the semigroup $\sgmG$ is also not a group.
\end{lem}
\begin{proof}
    If $\sgmG$ is a group, then $(A_i^{-1}, - A_i^{-1} \ba_i) = (A_i, \ba_i)^{-1} \in \sgmG$ for all $i$.
    Therefore, $A_i^{-1} \in \langle A_1, \ldots, A_K \rangle$ for all $i$.
    Thus $\langle A_1, \ldots, A_K \rangle$ is a group.
\end{proof}

Suppose now that $G$ is a group.
The key idea of solving the Group Problem is the following dichotomy known as the \emph{Tits alternative}.
\begin{thm}[{Tits alternative~\cite[Theorem~1]{TITS1972250}, effective version~\cite[Theorem~1.1]{beals1999algorithms}}]\label{thm:Tits}
    For any $n$, given a finitely generated subgroup $G$ of $\SL(n, \Z)$, exactly one of the following is true:
    \begin{enumerate}[label = (\roman*)]
        \item $G$ contains a non-abelian free subgroup.
        \item $G$ is virtually solvable.
    \end{enumerate}
    Furthermore, given a set of generators of $G$, it is decidable in PTIME which of the two cases is true.
\end{thm}

In case of $G$ containing a non-abelian free subgroup, we will prove the following Proposition~\ref{prop:nonsimid}, which shows that the semigroup $\sgmG$ must be a group.

\begin{restatable}{prop}{propnonsimid}\label{prop:nonsimid}
    Let $\mG = \{(A_1, \ba_1), \ldots, (A_K, \ba_K)\}$ be a set of elements of $\SA$, such that the semigroup $G \coloneqq \langle A_1, \ldots, A_K \rangle$ is a group.
    If $G$ contains a non-abelian free subgroup, then $\sgmG$ is a group.
\end{restatable}

The proof of Proposition~\ref{prop:nonsimid} is highly non-trivial and will be given in Section~\ref{sec:nonsim}.
The proof is mainly geometric - it consists of analysing the action of $\SA$ on the lattice $\Z^2$.

In case of $G$ being virtually solvable, we will prove the following Proposition~\ref{prop:simid} which refines the Tits alternative.\footnote{The semigroup $\sgmG$ is contained in the semidirect product $\Z^2 \rtimes G \leq \SA$, which is virtually solvable if $G$ is virtually solvable. A recent result by Bodart and Dong~\cite{bodart2024identity} shows that the Identity Problem and the Group Problem are decidable in all virtually solvable matrix groups over algebraic numbers. This notably includes virtually solvable subgroups of $\SA$. However, it is not clear whether the algorithm of Bodart and Dong runs in \emph{elementary} time, let alone PTIME.}
In particular, it shows that a virtually solvable subgroup of $\SL(2, \Z)$ has a relatively simple structure: it is either trivial, or it contains a torsion element, or it is infinite cyclic.

\begin{restatable}{prop}{propsimid}\label{prop:simid} 
    Let $\mG = \{(A_1, \ba_1), \ldots, (A_K, \ba_K)\}$ be a set of elements of $\SA$, such that the semigroup $G \coloneqq \langle A_1, \ldots, A_K \rangle$ is a group.
    If $G$ is virtually solvable, then exactly one of the following six conditions holds:
    \begin{enumerate}[label = (\roman*)]
        \item $G$ is the trivial group.
        \item $G$ contains a torsion element.
        \item $G = \langle A \rangle_{grp}$, where $A$ is a twisted inversion.
        \item $G = \langle A \rangle_{grp}$, where $A$ is a shear.
        \item $G = \langle A \rangle_{grp}$, where $A$ is an inverting scale.
        \item $G = \langle A \rangle_{grp}$, where $A$ is a positive scale.
    \end{enumerate}
    Furthermore, in cases (ii), (iii) and (v), the semigroup $\sgmG$ is a group.
    Overall, it is decidable in PTIME whether $\sgmG$ is a group.
\end{restatable}

Proposition~\ref{prop:simid} will be proved in Section~\ref{sec:sim}.
The proof will be mainly algebraic - it consists of analysing the structure of sub-semigroups of a virtually solvable group. 
Lemma~\ref{lem:isgrp}-Proposition~\ref{prop:simid} yield the decidability of the Group Problem (and consequently, the Identity Problem) in $\SA$.
An overview of the procedure is given in Algorithm~\ref{alg:GP}.
The justification of each step is given in parentheses with reference to the corresponding lemmas or propositions.

\begin{restatable}{thm}{thmid}\label{thm:id}
    The Group Problem and the Identity Problem in $\SA$ are NP-complete.
\end{restatable}
\begin{proof}
    The NP-hard lower bounds come from the NP-completeness of both problems in the subgroup $\SL(2, \Z) \cong \{(A, \bzer) \mid A \in \SL(2, \Z)\} \leq \SA$ (see Theorem~\ref{thm:SMPSL2}).
    
    To show decidability and the NP upper bounds, we first solve the Group Problem.
    Let $\mG = \{(A_1, \ba_1), \ldots, (A_K, \ba_K)\}$ be a set in $\SA$.
    As a first step, we can check whether $\langle A_1, \ldots, A_K \rangle$ is a group in NP (Theorem~\ref{thm:SMPSL2}).
    If $\langle A_1, \ldots, A_K \rangle$ is not a group, then $\sgmG$ is not a group by Lemma~\ref{lem:isgrp}.
    
    Suppose now that $\langle A_1, \ldots, A_K \rangle$ is a group.
    As the second step, we check in PTIME whether $\langle A_1, \ldots, A_K \rangle$ contains a non-abelian free subgroup using Theorem~\ref{thm:Tits}.
    If \linebreak[5] $\langle A_1, \ldots, A_K \rangle$ contains a non-abelian free subgroup, then the Group Problem has a positive answer by Proposition~\ref{prop:nonsimid}.
    Otherwise, $\langle A_1, \ldots, A_K \rangle$ is virtually solvable, and we can decide the Group Problem in PTIME using Proposition~\ref{prop:simid}.
    In total, the Group Problem for $\mG$ can be decided in NP.
    
    Next we solve the Identity Problem.
    Note that by Lemma~\ref{lem:grptoid}, $I \in \sgmG$ if and only if there exists a non-empty subset $\mH$ of $\mG$ such that the semigroup $\langle \mH \rangle$ is a group.
    Therefore, it suffices to guess a non-empty subset $\mH$ of $\mG$, and check whether $\langle \mH \rangle$ is a group.
    This can be done in NP using the above procedure for the Group Problem.
\end{proof}

\begin{algorithm}[ht!]
\caption{Deciding the Group Problem for a subset of $\SA$.}
\label{alg:GP}
\begin{description}
\item[Input] 
A subset $\mathcal{G} = \{(A_1, \ba_1), \ldots, (A_K, \ba_K)\}$ of $\SA$.
\item[Output] \textbf{True} or \textbf{False}.
\end{description}
\begin{enumerate}[label = \arabic*.]
    \item Decide whether the semigroup $\langle A_1, \ldots, A_K \rangle$ is a group by Theorem~\ref{thm:SMPSL2}.
    If not a group, return \textbf{False}. (Lemma~\ref{lem:isgrp})
    \item Let $G \coloneqq \langle A_1, \ldots, A_K \rangle \leq \SL(2, \Z)$, decide for $G$ which case of Theorem~\ref{thm:Tits} is true.
    \item If $G$ contains a non-abelian free subgroup, return \textbf{True}. (Proposition~\ref{prop:nonsimid})
    \item If $G$ is virtually solvable, decide which case of Proposition~\ref{prop:simid} is true using Lemma~\ref{lem:decisoZ}.
    \begin{enumerate}[label = (\roman*)]
        \item If $G$ is trivial, decide whether $n_1 \ba_1 + \cdots + n_K \ba_K = \bzer$ has a solution over $\Zpp^K$.
        If yes, return \textbf{True}, otherwise return \textbf{False}. (Proposition~\ref{prop:triv})
        \item If $G$ contains a torsion element, return \textbf{True}. (Proposition~\ref{prop:torid})
        \item If $G = \langle A \rangle_{grp}$, where $A$ is a twisted inversion, return \textbf{True}. (Proposition~\ref{prop:rev})
        \item If $G = \langle A \rangle_{grp}$, where $A$ is a shear, compute the set $\varphi(\mG)$ defined by Equation~\eqref{eq:defphi}.
        Decide whether $\langle \varphi(\mG) \rangle$ is a group by Theorem~\ref{thm:H3}.
        If yes, return \textbf{True}; if not, return \textbf{False}. (Corollary~\ref{cor:shear})
        \item If $G = \langle A \rangle_{grp}$, where $A$ is a inverting scale, return \textbf{True}. (Proposition~\ref{prop:le0})
        \item If $G = \langle A \rangle_{grp}$, where $A$ is a positive scale, compute the set $\mS$ defined by Equation~\eqref{eq:defS}.
        Decide whether the condition in Proposition~\ref{prop:idtocells} is satisfied for $\mS$.
        If yes, return \textbf{True}; if not, return \textbf{False}. (Corollary~\ref{cor:simdec})
    \end{enumerate}
\end{enumerate}
\end{algorithm}

\section{Non-abelian free subgroup}\label{sec:nonsim}
In this section we prove Proposition~\ref{prop:nonsimid}:
\propnonsimid*

The proof depends on two lemmas concerning the effect of ``pumping'' a word $(A, \ba) (B, \bb)$, where $(A, \ba) \cdot (B, \bb) = (I, \bx)$ for some $\bx \in \Z^2$.
Let $A$ be a scale with $\Lat(A) = \{V, W\}$. 
Since $V, W$ are distinct one dimensional subspaces of $\R^2$, every element $\bx \in \R^2$ can be written uniquely as $\bx = \bx_V + \bx_W$, where $\bx_V \in V, \bx_W \in W$.
We will adopt this notation in the following lemma.
For an element $\byy \in \R^2$, $\byy_V, \byy_W$ are defined similarly.
We use $\| \cdot \|$ to denote the Euclidean norm.

\begin{restatable}{lem}{lemscalelim}\label{lem:scalelim}
    Let $(A, \ba), (B, \bb)$ be elements of $\SA$ such that $(A, \ba) \cdot (B, \bb) = (I, \bx)$ for some $\bx \in \Z^2$.
    Suppose $A$ is a scale; denote by $V, W$ the elements of $\Lat(A)$, and suppose $\bx \notin V \cup W$.
    Let $\bv$ be any non-zero vector in the subspace $V$.
    
    Then for every $\varepsilon \in (0, 1)$, there exists a word $w \in \{(A, \ba), (B, \bb)\}^*$, such that $(A, \ba) \cdot w \cdot (B, \bb) = (I, \byy)$, where $\byy \in \Z^2$ satisfies
    \begin{equation}\label{eq:lim}
        1 - \frac{|\bv^{\top} \byy|}{\|\bv\| \|\byy\|} < \varepsilon, \quad \byy_V^{\top} \bx_V > 0, \quad \byy_W^{\top} \bx_W > 0.
    \end{equation}
    In other words, the acute angle $\theta$ between $\byy$ and $V$ satisfies $1 - \cos \theta < \varepsilon$.
    Also, $\byy$ and $\bx$ lie in the same cone out of the four cut out by $V$ and $W$.
    See Figure~\ref{fig:scalelim} for an illustration.
\end{restatable}

    \begin{figure}[ht!]
        \centering
        \begin{minipage}[t]{.45\textwidth}
            \centering
            \includegraphics[width=\textwidth, keepaspectratio, trim={4.5cm 0cm 4cm 0.5cm},clip]{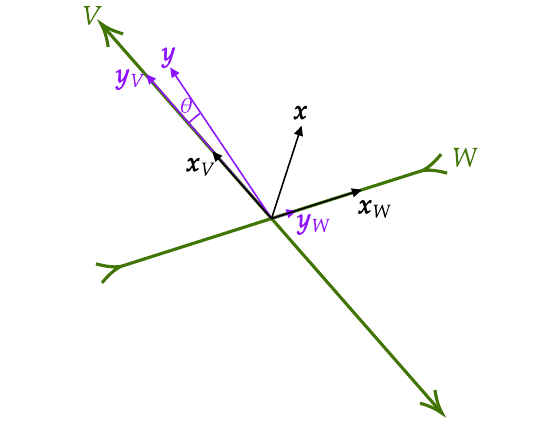}
            \caption{Illustration for Lemma~\ref{lem:scalelim}.}
            \label{fig:scalelim}
        \end{minipage}
        \hfill
        \begin{minipage}[t]{0.45\textwidth}
            \centering
            \includegraphics[width=\textwidth, keepaspectratio, trim={4.5cm 0.5cm 4cm 0cm},clip]{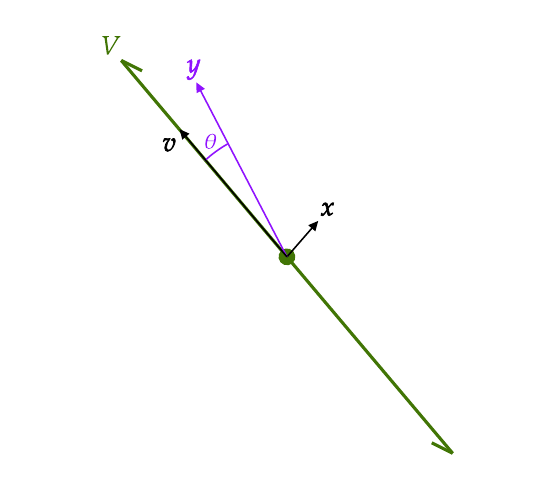}
            \caption{Illustration for Lemma~\ref{lem:shearlim}.}
            \label{fig:shearlim}
        \end{minipage}
    \end{figure}

\begin{proof}
    Since $(A, \ba) \cdot (B, \bb) = (I, \bx)$, we have $B = A^{-1}$ and $\bx = A \bb + \ba$. In particular, we have $\Lat(A) = \Lat(B)$.

    Let $\lambda$ be the eigenvalue of $A$ associated to the invariant subspace $V$, then $\lambda^{-1}$ is the eigenvalue associated to the invariant subspace $W$.    
    Then we have $A^n \bx = \lambda^n \bx_V + \lambda^{-n} \bx_W$ for every integer $n$.
    Consider two cases.
    \begin{enumerate}
        \item $V$ is a stretching direction.
        In this case, $|\lambda| > 1$.
        Let $m > 0$ be a positive integer, and let $w \coloneqq (A, \ba)^{m-1} (B, \bb)^{m-1}$. 
        Consider the product 
        \begin{align}\label{eq:powm}
            & (A, \ba) \cdot w \cdot (B, \bb) \nonumber \\
            = \;&  (A, \ba)^m (B, \bb)^m \nonumber \\
            = \;&  (I, (I + A + \cdots + A^{(m-1)})(A \bb + \ba)) \nonumber \\
            = \;&  (I, (I + A + \cdots + A^{(m-1)})\bx) \nonumber \\
            = \;& \left(I, \sum_{i = 0}^{m-1} \lambda^i \bx_V + \sum_{i = 0}^{m-1} \lambda^{-i} \bx_W \right)
        \end{align}
        Let $\byy \coloneqq \sum_{i = 0}^{m-1} \lambda^i \bx_V + \sum_{i = 0}^{m-1} \lambda^{-i} \bx_W$, then $\byy_V = \sum_{i = 0}^{m-1} \lambda^i \bx_V$, $\byy_W = \sum_{i = 0}^{m-1} \lambda^{-i} \bx_W$.
        Since $\bx \notin V \cup W$, we have $\bx_V \neq \bzer, \bx_W \neq \bzer$.
        When $m$ is odd, we have $\sum_{i = 0}^{m-1} \lambda^i > 0$, $\sum_{i = 0}^{m-1} \lambda^{-i} > 0$, so $\byy_V^{\top} \bx_V > 0$ and $\byy_W^{\top} \bx_W > 0$.

        Note that we always have $\frac{|\bv^{\top} \byy|}{\|\bv\| \|\byy\|} \leq 1$ by the Cauchy-Schwarz inequality.
        We then show that when $m$ tends towards infinity, the value $\frac{|\bv^{\top} \byy|}{\|\bv\| \|\byy\|}$ will tend to one.
        Indeed,
        \begin{align*}
            \frac{|\bv^{\top} \byy|}{\|\bv\| \|\byy\|} = \; & \frac{\left|\sum_{i = 0}^{m-1} \lambda^i \bv^{\top} \bx_V + \sum_{i = 0}^{m-1} \lambda^{-i} \bv^{\top} \bx_W\right|}{\|\bv\| \left\|\sum_{i = 0}^{m-1} \lambda^i \bx_V + \sum_{i = 0}^{m-1} \lambda^{-i} \bx_W \right\|} \\
            \geq \; & \frac{\left|\left(\sum_{i = 0}^{m-1} \lambda^i \right) \bv^{\top} \bx_V \right|}{\|\bv\| \left\|\sum_{i = 0}^{m-1} \lambda^i \bx_V + \sum_{i = 0}^{m-1} \lambda^{-i} \bx_W \right\|} 
             - \frac{\left|\left(\sum_{i = 0}^{m-1} \lambda^{-i}\right) \bv^{\top} \bx_W\right|}{\|\bv\| \left\|\sum_{i = 0}^{m-1} \lambda^i \bx_V + \sum_{i = 0}^{m-1} \lambda^{-i} \bx_W \right\|} \\
            = \; & \frac{\|\bv\| \|\bx_V\|}{\|\bv\| \left\|\bx_V + \lambda^{1 - m} \bx_W \right\|} 
             - \frac{\left|\bv^{\top} \bx_W\right|}{\|\bv\| \left\|\lambda^{m - 1}\bx_V + \bx_W \right\|}
        \end{align*}
        When $m \rightarrow \infty$, the right hand side tends towards $1-0 = 1$.
        This is because $|\lambda| > 1$ and $\bv \neq \bzer, \bx_V \neq \bzer$.
        Hence, for a large enough odd integer $m$, we have
        \begin{equation*}
            1 - \frac{|\bv^{\top} \byy|}{\|\bv\| \|\byy\|} < \varepsilon, \quad \byy_V^{\top} \bx_V > 0, \quad \byy_W^{\top} \bx_W > 0.
        \end{equation*}

        \item $V$ is a compressing direction.
        In this case, $|\lambda| < 1$.
        Let $m > 0$ be a positive integer, and let $w \coloneqq (B, \bb)^{m} (A, \ba)^{m}$. 
        Consider the product 
        \begin{align*}
            & (A, \ba) \cdot w \cdot (B, \bb) \\
            = \;&  (A, \ba) (B, \bb) (B, \bb)^{m-1} (A, \ba)^{m-1} (A, \ba) (B, \bb) \\
            = \;& \left(I, (2I + A^{-1} + \cdots + A^{-(m-1)})(A \bb + \ba)\right) \\
            = \;& \left(I, (2I + A^{-1} + \cdots + A^{-(m-1)})\bx \right) \\
            = \;& \left(I, \left(1 + \sum_{i = 0}^{m-1} \lambda^{-i}\right) \bx_V + \left(1 + \sum_{i = 0}^{m-1} \lambda^{i}\right) \bx_W \right)
        \end{align*}
        Let $\byy \coloneqq \left(1 + \sum_{i = 0}^{m-1} \lambda^{-i}\right) \bx_V + \left(1 + \sum_{i = 0}^{m-1} \lambda^{i}\right) \bx_W$, then $\byy_V = \left(1 + \sum_{i = 0}^{m-1} \lambda^{-i}\right) \bx_V$, $\byy_W = \left(1 + \sum_{i = 0}^{m-1} \lambda^{i}\right) \bx_W$.
        Since $\bx \notin V \cup W$, we have $\bx_V \neq \bzer$ and $\bx_W \neq \bzer$.
        When $m$ is odd, we have $\sum_{i = 0}^{m-1} \lambda^i > 0$, $\sum_{i = 0}^{m-1} \lambda^{-i} > 0$, so $\byy_V^{\top} \bx_V > 0, \byy_W^{\top} \bx_W > 0$.
        
        We then show that when $m$ tends towards infinity, the value $\frac{|\bv^{\top} \byy|}{\|\bv\| \|\byy\|} \leq 1$ will tend to one.
        Indeed,
        \begin{align*}
             \frac{\left|\bv^{\top} \byy\right|}{\|\bv\| \|\byy\|}
            = \; & \frac{\left|\left(1 + \sum_{i = 0}^{m-1} \lambda^{-i}\right) \bv^{\top} \bx_V + \left(1 + \sum_{i = 0}^{m-1} \lambda^{i}\right) \bv^{\top} \bx_W\right|}{\|\bv\| \left\|\left(1 + \sum_{i = 0}^{m-1} \lambda^{-i}\right) \bx_V + \left(1 + \sum_{i = 0}^{m-1} \lambda^{i}\right) \bx_W \right\|} \\
            \geq \; & \frac{\|\bv\| \|\bx_V\|}{\|\bv\| \left\|\bx_V + \frac{1 + \sum_{i = 0}^{m-1} \lambda^{i}}{1 + \sum_{i = 0}^{m-1} \lambda^{-i}} \bx_W \right\|} 
            - \frac{\left|\bv^{\top} \bx_W\right|}{\|\bv\| \left\|\frac{1 + \sum_{i = 0}^{m-1} \lambda^{-i}}{1 + \sum_{i = 0}^{m-1} \lambda^{i}} \bx_V + \bx_W \right\|}
        \end{align*}
        When $m \rightarrow \infty$, the right hand side tends towards $1 - 0 = 1$.
        This is because $|\lambda| < 1$ and $\bv \neq \bzer, \bx_V \neq \bzer$, so
        \[
        \lim_{m \rightarrow \infty} \frac{1 + \sum_{i = 0}^{m-1} \lambda^{i}}{1 + \sum_{i = 0}^{m-1} \lambda^{-i}} = 0.
        \]
        Hence, for a large enough odd integer $m$, we have
        \begin{equation*}
            1 - \frac{|\bv^{\top} \byy|}{\|\bv\| \|\byy\|} < \varepsilon, \quad \byy_V^{\top} \bx_V > 0, \quad \byy_W^{\top} \bx_W > 0.
        \end{equation*}
    \end{enumerate}
    Combining the two cases yields the desired result.
\end{proof}

A similar lemma can be proved for shears.
In this case we want the stronger condition $1 - \frac{\bv^{\top} \byy}{\|\bv\| \|\byy\|} < \varepsilon$ instead of $1 - \frac{|\bv^{\top} \byy|}{\|\bv\| \|\byy\|} < \varepsilon$.

\begin{restatable}{lem}{lemshearlim}\label{lem:shearlim}
    Let $(A, \ba), (B, \bb)$ be elements of $\SA$ such that $(A, \ba) \cdot (B, \bb) = (I, \bx)$ for some $\bx \in \Z^2$.
    Suppose $A$ is a shear, $\Lat(A) = \{V\}$, and $\bx \notin V$.
    Let $\bv$ be any non-zero vector in the subspace $V$.
    
    Then for every $\varepsilon \in (0, 1)$, there exists a word $w \in \{(A, \ba), (B, \bb)\}^*$, such that $(A, \ba) \cdot w \cdot (B, \bb) = (I, \byy)$, where $\byy \in \Z^2$ satisfies
    \begin{equation}\label{eq:shearlimA}
        1 - \frac{\bv^{\top} \byy}{\|\bv\| \|\byy\|} < \varepsilon,
    \end{equation}
    and $\byy$ and $\bx$ lie in same halfspace cut by $V$.
    In other words, the angle $\theta$ between $\byy$ and $\bv$ satisfies $1 - \cos \theta < \varepsilon$.
    See Figure~\ref{fig:shearlim} for an illustration.
\end{restatable}
\begin{proof}
    Since $(A, \ba) \cdot (B, \bb) = (I, \bx)$, we have $B = A^{-1}$ and $\bx = A \bb + \ba$.

    Let $W$ be the orthogonal space of $V$, and $\bw$ be a non-zero vector in $W$.
    Under the basis $\{\bv, \bw\}$, the matrix $A$ has the form 
    $
    \begin{pmatrix}
        1 & \mu \\
        0 & 1 \\
    \end{pmatrix}
    $,
    where $\mu \neq 0$.
    Then for every integer $n$, we have $A^n \bx = \bx + n \mu c \bv$, where $c$ is a scalar such that $c \bw = \bx_W$.
    Since $\bx \notin V$, we have $\bx_W \neq \bzer$, so $c \neq 0$.
    Consider two cases.
    \begin{enumerate}
        \item $\mu c > 0$.
        Let $m > 0$ be a positive integer, and let $w \coloneqq (A, \ba)^{m-1} (B, \bb)^{m-1}$. 
        Consider the product 
        \begin{align*}
            & (A, \ba) \cdot w \cdot (B, \bb) \\
            = \;& (A, \ba)^m (B, \bb)^m \\
            = \;& \left(I, (I + A + \cdots + A^{(m-1)})(A \bb + \ba)\right) \\
            = \;& \left(I, (I + A + \cdots + A^{(m-1)})\bx\right) \\
            = \;& \left(I, m\bx + \frac{m(m-1)}{2} \mu c \bv\right).
        \end{align*}
        Let $\byy \coloneqq m\bx + \frac{m(m-1)}{2} \mu c \bv$, then $\byy$ and $\bx$ lie in same halfspace cut by $V$.

        We then show that when $m$ tends towards infinity, $\frac{\bv^{\top} \byy}{\|\bv\| \|\byy\|}$ will tend to one.
        Indeed,
        \begin{align*}
            \frac{\bv^{\top} \byy}{\|\bv\| \|\byy\|} = \; & \frac{m\bv^{\top} \bx + \frac{m(m-1)}{2} \mu c \bv^{\top} \bv}{\|\bv\| \left\|m\bx + \frac{m(m-1)}{2} \mu c \bv \right\|} \\
            = \; & \frac{\bv^{\top} \bx}{\|\bv\| \left\|\bx + \frac{(m-1)}{2} \mu c \bv \right\|} + \frac{\frac{(m-1)}{2} \mu c \|\bv\|}{\left\|\bx + \frac{(m-1)}{2} \mu c \bv \right\|}.
        \end{align*}
        When $m \rightarrow \infty$, this expression tends towards $0 + 1 = 1$, because $\mu c > 0$ and $\bv \neq \bzer$.
        Hence, for a large enough integer $m$, we have
        $
            1 - \frac{\bv^{\top} \byy}{\|\bv\| \|\byy\|} < \varepsilon.
        $

        \item $\mu c < 0$.
        Let $m > 0$ be a positive integer, and let $w \coloneqq (B, \bb)^{m} (A, \ba)^{m}$. 
        Consider the product 
        \begin{align*}
            & (A, \ba) \cdot w \cdot (B, \bb) \\
            = \; &  (A, \ba) (B, \bb) (B, \bb)^{m-1} (A, \ba)^{m-1} (A, \ba) (B, \bb) \\
            = \; & \left(I, (2I + A^{-1} + \cdots + A^{-(m-1)})(A \bb + \ba)\right) \\
            = \; & \left(I, (2I + A^{-1} + \cdots + A^{-(m-1)})\bx\right) \\
            = \; & \left(I, (m+1)\bx - \frac{m(m-1)}{2} \mu c \bv\right).
        \end{align*}
        Let $\byy \coloneqq (m+1)\bx - \frac{m(m-1)}{2} \mu c \bv$, then $\byy$ and $\bx$ lie in same halfspace cut by $V$.

        We then show that when $m$ tends towards infinity, $\frac{\bv^{\top} \byy}{\|\bv\| \|\byy\|}$ will tend to one.
        Indeed,
        \begin{align*}
            \frac{\bv^{\top} \byy}{\|\bv\| \|\byy\|} = \; & \frac{(m+1)\bv^{\top} \bx - \frac{m(m-1)}{2} \mu c \bv^{\top} \bv}{\|\bv\| \left\|(m+1)\bx - \frac{m(m-1)}{2} \mu c \bv \right\|} \\
            = \; & \frac{ \bv^{\top} \bx}{\|\bv\| \left\|\bx - \frac{m(m-1)}{2(m+1)} \mu c \bv \right\|} + \frac{- \frac{m(m-1)}{2(m+1)} \mu c \|\bv\|}{\left\|\bx - \frac{m(m-1)}{2(m+1)} \mu c \bv \right\|}.
        \end{align*}
        When $m \rightarrow \infty$, this expression tends towards $0 + 1 = 1$, because $\mu c < 0$ and $\bv \neq \bzer$.
        Hence, for a large enough integer $m$, we have
        $
            1 - \frac{\bv^{\top} \byy}{\|\bv\| \|\byy\|} < \varepsilon.
        $ \qedhere
    \end{enumerate}
\end{proof}

We now prove Proposition~\ref{prop:nonsimid}.
\propnonsimid*
\begin{proof}
    Suppose that the group $\langle A_1, \ldots, A_K \rangle$ contains a non-abelian free subgroup.
    In particular, it contains a subgroup isomorphic to the free group $F_2$ of two generators.
    Let $A, B$ be elements of $\langle A_1, \ldots, A_K \rangle$ that generate this free group. 
    Since $A, B$ are non-torsion, they are twisted inversions, scales, or shears.
    Hence, $A^2, B^2 \in \langle A_1, \ldots, A_K \rangle$ are \emph{positive} scales or shears.
    Since $A$ and $B$ generate a non-abelian free group, $A^2$ and $B^2$ also generate a non-abelian free group.
    Therefore, we can replace $A, B$ by $A^2, B^2$, and without loss of generality suppose $A$ and $B$ to be positive scales or shears.
    
    Since $\langle A_1, \ldots, A_K \rangle$ is a group, $A$ and $B$ can be represented by full-image words over the alphabet $\{A_1, \ldots, A_K\}$ due to Lemma~\ref{lem:grpword}.
    Hence, let $(A, \ba)$ and $(B, \bb)$ be elements in $\sgmG$ represented by full-image words.

    Note that $A, B$ are not simultaneously triangularizable, otherwise the group they generate is isomorphic to a subgroup of the group of upper-triangular matrices over complex numbers, which is solvable~\cite{beals1999algorithms}.
    This contradicts that fact that $A, B$ generate a non-abelian free group (see Theorem~\ref{thm:Tits}).
    Therefore, $\Lat(A) \cap \Lat(B) = \emptyset$.
    There are three cases to consider: 
    \begin{enumerate}[nosep, label = (\arabic*)]
        \item Both $A$ and $B$ are positive scales.
        \item One of $A$ and $B$ is a shear.
        \item Both $A$ and $B$ are shears.
    \end{enumerate}
    
    \textbf{Case 1. Both $A$ and $B$ are positive scales.}
    Denote $Y = A^{-1} B^{-1} \in \langle A_1, \ldots, A_K \rangle$, we have $AYB = I$.
    Let $\byy \in \Z^2$ be such that $(Y, \byy) \in \sgmG$, and let $\bx \in \Z^2$ be such that $(A, \ba)(Y, \byy)(B, \bb) = (I, \bx)$.
    If $\bx = \bzer$, then $(I, \bzer)$ can be represented as a full-image word over $\mG$, since $(A, \ba)$ can.
    In this case, $\sgmG$ is a group by Lemma~\ref{lem:grpword}.
    Otherwise, there are two subcases.
    Starting from any one-dimensional subspace $S$, rotate $S$ counter-clockwise and consider its sequence of encounters with $\Lat(A)$ and $\Lat(B)$ before finishing a half circle.
    Either the sequence of encounters is a cyclic permutation of $(\Lat(A), \Lat(A), \Lat(B), \Lat(B))$, or it is a cyclic permutation of $(\Lat(A), \Lat(B), \Lat(A), \Lat(B))$.

    \begin{enumerate}[label = (\alph*)]
        \item First consider the case of $(\Lat(A)$, $\Lat(A)$, $\Lat(B)$, $\Lat(B))$.
        An illustration for the proof of this case is shown in Figure~\ref{fig:AABB}.
        Denote by $V_A, W_A, V_B, W_B$ these subspaces in this order.
        Since $\bx \neq \bzer$, we have $\bx \notin V_A \cup W_A$ or $\bx \notin V_B \cup W_B$.
        Without loss of generality suppose $\bx \notin V_B \cup W_B$.

    \begin{figure}[ht!]
        \centering
        \begin{minipage}[t]{.45\textwidth}
            \centering
            \includegraphics[width=\textwidth, keepaspectratio, trim={4.5cm 0.7cm 4cm 0cm},clip]{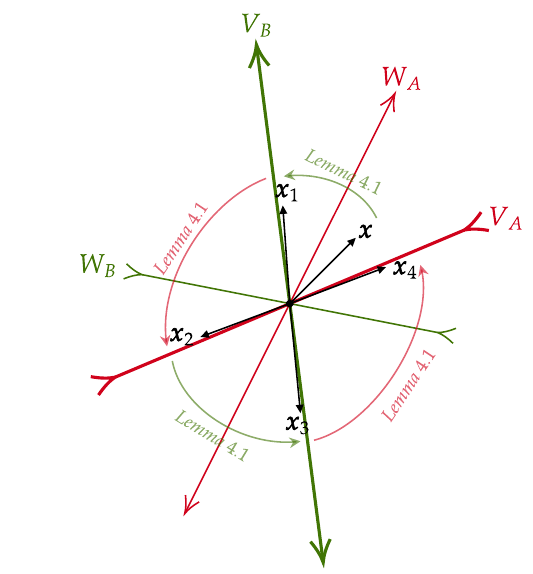}
            \caption{Illustration for case 1(a).}
            \label{fig:AABB}
        \end{minipage}
        \hfill
        \begin{minipage}[t]{0.45\textwidth}
            \centering
            \includegraphics[width=\textwidth, keepaspectratio, trim={4.5cm 0.5cm 4cm 0cm},clip]{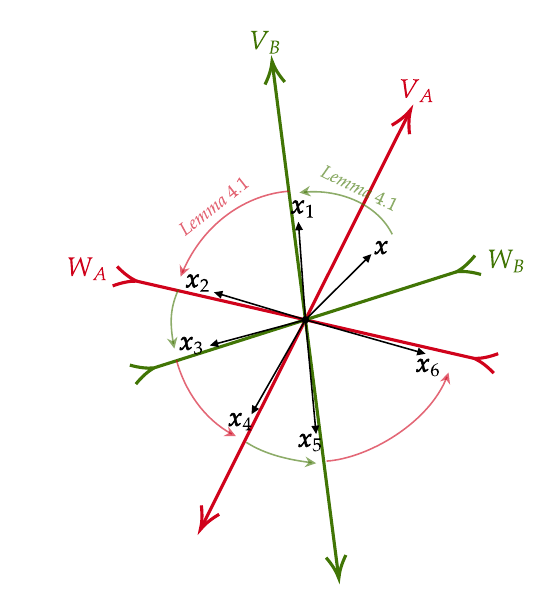}
            \caption{Illustration for case 1(b).}
            \label{fig:ABAB}
        \end{minipage}
        \centering
        \begin{minipage}[t]{.45\textwidth}
            \centering
            \includegraphics[width=\textwidth, keepaspectratio, trim={4.5cm 0.5cm 3.5cm 0cm},clip]{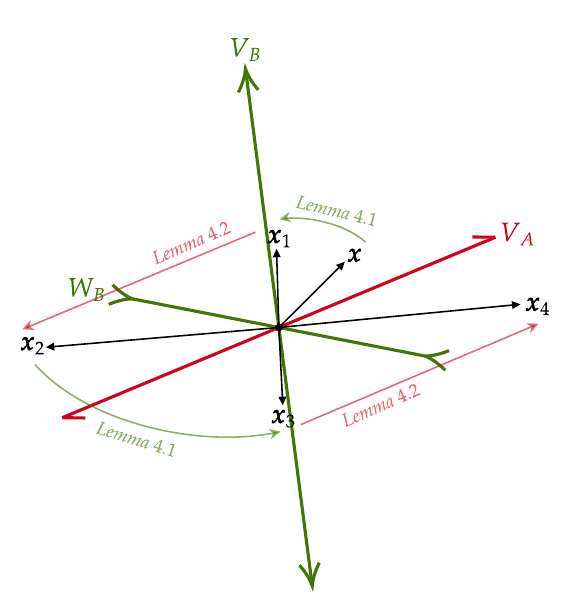}
            \caption{Illustration for case 2.}
            \label{fig:ABB}
        \end{minipage}
        \hfill
        \begin{minipage}[t]{0.45\textwidth}
            \centering
            \includegraphics[width=\textwidth, keepaspectratio, trim={4.5cm 0.7cm 3.5cm 0cm},clip]{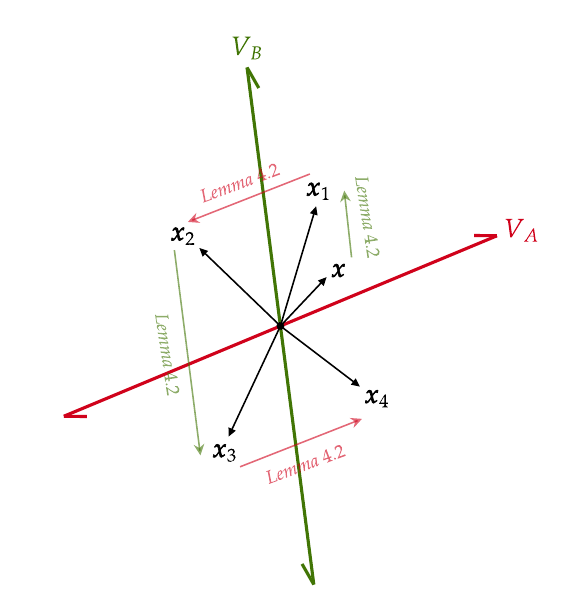}
            \caption{Illustration for case 3.}
            \label{fig:AB}
        \end{minipage}
    \end{figure}
    
        Let $\varepsilon > 0$.
        Apply Lemma~\ref{lem:scalelim} to $\varepsilon$, to the elements $(A', \ba') \coloneqq (A, \ba) (Y, \byy)$ and $(B, \bb)$, and to the subspaces $V_B, W_B$ of $\Lat(B) = \Lat(A')$.
        (Note that $A' = B^{-1}$, so $A'$ is a positive scale and $\Lat(A') = \Lat(B)$.)
        Since $(A', \ba')(B, \bb) = (A, \ba)(Y, \byy)(B, \bb) = (I, \bx)$, Lemma~\ref{lem:scalelim} shows there exists an element $w \in \sgmG$ such that $(A', \ba') \cdot w \cdot (B, \bb) = (I, \bx_1)$, with 
        \begin{equation}\label{eq:x1}
            1 - \frac{|\bv_B^{\top} \bx_1|}{\|\bv_B\| \|\bx_1\|} < \varepsilon, \; \left(\bx_1\right)_{V_B}^{\top} \bx_{V_B} > 0, \; \left(\bx_1\right)_{W_B}^{\top} \bx_{W_B} > 0.
        \end{equation}
        Here, $\bv_B$ is a non-zero vector in $V_B$.
        In other words, when $\varepsilon$ is sufficiently small, the acute angle between $\bx_1$ and $V_B$ is also sufficiently small.
        Also, $\bx_1$ and $\bx$ lie in the same cone out of the four cut by $V_B$ and $W_B$.
        This gives us an element $(Y_1, \byy_1) \coloneqq (Y, \byy) \cdot w \in \sgmG$ such that $(A, \ba)(Y_1, \byy_1)(B, \bb) = (I, \bx_1)$ with $\bx_1$ satisfying \eqref{eq:x1}, see Figure~\ref{fig:AABB}.

        Next, apply Lemma~\ref{lem:scalelim} to $\varepsilon$, to the elements $(A, \ba)$ and $(B', \bb') \coloneqq (Y_1, \byy_1)(B, \bb)$ of $\SA$, and to the subspaces $V_A, W_A$ in $\Lat(A)$.
        Same as above, we obtain an element $(Y_2, \byy_2) \in \mG$ such that $(A, \ba)(Y_2, \byy_2)(B, \bb) = (I, \bx_2)$ with $\bx_2$ satisfying
        \begin{equation}\label{eq:x2}
            1 - \frac{|\bv_A^{\top} \bx_2|}{\|\bv_A\| \|\bx_2\|} < \varepsilon, \; \left(\bx_2\right)_{V_A}^{\top} \bx_{V_A} > 0, \; \left(\bx_2\right)_{W_A}^{\top} \cdot \bx_{W_A} > 0.
        \end{equation}
        Here, $\bv_A$ is a non-zero vector in $V_A$.
        In other words, when $\varepsilon$ is sufficiently small, the acute angle between $\bx_2$ and $V_A$ is also sufficiently small.
        Hence, one can take $\varepsilon$ such that $\bx_2$ and $- \bx_1$ lie in the same cone out of the four cut by $V_B$ and $W_B$.
        Furthermore, $\bx_2$ and $\bx_1$ lie in the same cone out of the four cut by $V_A$ and $W_A$.

        We follow this pattern and apply Lemma~\ref{lem:scalelim} again on $(A', \ba') \coloneqq (A, \ba) (Y_2, \byy_2)$ and $(B, \bb)$, and to the subspaces $V_B, W_B$.
        This yields an element $(Y_3, \byy_3) \in \sgmG$ such that $(A, \ba)(Y_3, \byy_3)(B, \bb) = (I, \bx_3)$ with $\bx_3$ very close to $V_B$, but lies in the same cone cut by $V_B$ and $W_B$ as $\bx_2$.

        Finally we apply Lemma~\ref{lem:scalelim} again on $(A, \ba)$ and $(B', \bb') \coloneqq (Y_3, \byy_3)(B, \bb)$, and to the subspaces $V_A, W_A$.
        This yields $(Y_4, \byy_4) \in \sgmG$ such that $(A, \ba)(Y_4, \byy_4)(B, \bb) = (I, \bx_4)$ with $\bx_4$ very close to $V_A$, but lies in the same cone cut by $V_A$ and $W_A$ as $\bx_3$.

        When $\varepsilon$ is small enough, the angles of the vectors $\bx_1, \bx_2, \bx_3, \bx_4$ are sufficiently close to $V_B, V_A, V_B, V_A$, with opposing directions.
        Hence, they generate $\Q^2$ as a $\Qp$-cone.
        In other words, there exist \emph{positive} integers $n_1, n_2, n_3, n_4$ such that $n_1\bx_1 + n_2\bx_2 + n_3\bx_3 + n_4\bx_4 = \bzer$.
        Therefore,
        \[
        (I, \bx_1)^{n_1} (I, \bx_2)^{n_2} (I, \bx_3)^{n_3} (I, \bx_4)^{n_4} = (I, \bzer).
        \]
        Since the element $(A, \ba)$ can be represented as a full-image word over $\mG$, the element $(I, \bx_1)$ and hence $(I, \bzer)$ can be represented by a full-image word as well.
        Therefore, $\sgmG$ is a group by Lemma~\ref{lem:grpword}.
        
        \item Next consider the case $(\Lat(A)$, $\Lat(B)$, $\Lat(A)$, $\Lat(B))$. 
        Denote by $V_A, V_B, W_A, W_B$ these subspaces in this order.
        Without loss of generality suppose $\bx \notin V_B \cup W_B$.
        The strategy is exactly the same as the previous case (see Figure~\ref{fig:ABAB}).
        However, in the present case, we need to apply Lemma~\ref{lem:scalelim} for a total of six times, where $V$ will be the subspaces $V_B, W_A, W_B, V_A, V_B, W_A$ respectively.

        In this way, we obtain $(I, \bx_1)$, $(I, \bx_2)$, $(I, \bx_3)$, $(I, \bx_4)$, $(I, \bx_5)$, $(I, \bx_6)$ with $\bx_1, \ldots, \bx_6$ generating $\Q^2$ as a $\Qp$-cone.
        Hence, there exist positive integers $n_1, \ldots, n_6$ such that
        \[
        (I, \bx_1)^{n_1} \cdots (I, \bx_6)^{n_6} = (I, \bzer).
        \]
        (In fact, the four vectors $\bx_1, \bx_2, \bx_5, \bx_6$ would already suffice).
        Similarly, since the element $(A, \ba)$ (and hence $(I, \bzer)$) can be represented as a full-image word over $\mG$, the semigroup $\sgmG$ is a group by Lemma~\ref{lem:grpword}.
    \end{enumerate}
    This concludes case 1.
    
    \textbf{Case 2. One of $A$ and $B$ is a shear.}
    The approach is similar to the previous cases, but we have to apply Lemma~\ref{lem:scalelim} and Lemma~\ref{lem:shearlim} alternately.
    See Figure~\ref{fig:ABB} for an illustration.
    
    Without loss of generality, let $A$ be the shear and $B$ be the positive scale.
    Starting from any one-dimension subspace $S$, rotate $S$ counter-clockwise and consider its sequence of encounters with $\Lat(A)$ and $\Lat(B)$ before finishing a full cycle.
    The sequence of encounters must be a cyclic permutation of $(\Lat(A), \Lat(B), \Lat(B))$.
    Denote by $V_A, V_B, W_B$ be these subspaces in this order.
    Without loss of generality suppose $\bx \notin V_B \cup W_B$, the case where $\bx \notin V_A$ is analogous.
    
    For a small enough $\varepsilon$, apply Lemma~\ref{lem:scalelim} to $\bx$ to obtain $\bx_1$ that is sufficiently close to $V_B$; then apply Lemma~\ref{lem:shearlim} to $\bx_1$ to obtain $\bx_2$ that is sufficiently close to $V_A$ and therefore lies in the same cone cut by $V_B$ and $W_B$ as $-\bx_1$; then apply Lemma~\ref{lem:scalelim} to $\bx_2$ to obtain $\bx_3$ that is sufficiently close to $V_B$ and therefore lies in a different halfspace cut by $V_A$ than $\bx_2$; finally, apply Lemma~\ref{lem:shearlim} to $\bx_3$ to obtain $\bx_4$ that is sufficiently close to $V_A$ and lies in the same cone cut by $V_B$ and $W_B$ as $-\bx_3$.
    In this way, we obtain $(I, \bx_1), (I, \bx_2), (I, \bx_3), (I, \bx_4)$ with $\bx_1, \ldots, \bx_4$ generating $\Q^2$ as a $\Qp$-cone.
    Hence, there exist positive integers $n_1, \ldots, n_4$ such that
    \[
    (I, \bx_1)^{n_1} \cdots (I, \bx_4)^{n_4} = (I, \bzer).
    \]
    Since the element $(A, \ba)$ (and hence $(I, \bzer)$) can be represented as a full-image word over $\mG$, the semigroup $\sgmG$ is a group.

    \textbf{Case 3. Both $A$ and $B$ are shears.}
    The approach is similar to the previous cases, but we have to apply Lemma~\ref{lem:shearlim} only.
    See Figure~\ref{fig:AB} for an illustration.
    
    Denote by $V_A, V_B$ respectively the elements of $\Lat(A)$ and $\Lat(B)$.
    Without loss of generality suppose $\bx \notin V_B$, the case where $\bx \notin V_A$ is analogous.
    For a small enough $\varepsilon$, apply Lemma~\ref{lem:shearlim} to $\bx$ to obtain $\bx_1$ that is sufficiently close to $V_B$; then apply Lemma~\ref{lem:shearlim} again to $\bx_1$ to obtain $\bx_2$ that is sufficiently close to $V_A$ and lies in a different halfspace cut by $V_B$ than $\bx_1$; then apply Lemma~\ref{lem:shearlim} to $\bx_2$ to obtain $\bx_3$ that is sufficiently close to $V_B$ and lies in a different halfspace cut by $V_A$ than $\bx_2$; finally, apply Lemma~\ref{lem:shearlim} to $\bx_3$ to obtain $\bx_4$ that is sufficiently close to $V_A$ and lies in a different halfspace cut by $V_B$ than $\bx_3$.
    In this way, we obtain $(I, \bx_1), (I, \bx_2), (I, \bx_3), (I, \bx_4)$ with $\bx_1, \ldots, \bx_4$ generating $\Q^2$ as a $\Qp$-cone.
    Hence, there exist positive integers $n_1, \ldots, n_4$ such that
    \[
    (I, \bx_1)^{n_1} \cdots (I, \bx_4)^{n_4} = (I, \bzer).
    \]        
    Since the element $(A, \ba)$ (and hence $(I, \bzer)$) can be represented as a full-image word over $\mG$, the semigroup $\sgmG$ is a group.
\end{proof}

\section{Virtual solvability}\label{sec:sim}
In this section we prove Proposition~\ref{prop:simid}:
\propsimid*

As in the statement of the proposition, we fix a set $\mG = \{(A_1, \ba_1), \ldots, (A_K, \ba_K)\}$ of elements in $\SA$, such that $H \coloneqq \langle A_1, \ldots, A_K \rangle$ is a virtually solvable group.

\begin{lem}\label{lem:solvZ}
    Let $H$ be a finitely generated virtually solvable subgroup of $\SL(2, \Z)$. Then $H$ is either finite, or it contains a finite index subgroup $H'$ isomorphic to $\Z$.
\end{lem}
\begin{proof}
    By Theorem~\ref{thm:vf}, let $F \leq \SL(2, \Z)$ be a free subgroup over two generators, such that the index $[\SL(2, \Z) : F]$ is finite.
    Then $[H : F \cap H] \leq [\SL(2, \Z) : F]$, so the group $H' \coloneqq F \cap H$ is of finite index in $H$ and hence finitely generated.
    But $H'$ is a subgroup of the free group $F$, so it must be free by Theorem~\ref{thm:freesub}.
    On the other hand, $H'$ is a subgroup of the virtually solvable group $H$, so $H'$ is virtually solvable.
    Since $H' \leq \SL(2, \Z)$ is virtually solvable and free, it must be abelian by Theorem~\ref{thm:Tits}. 
    Therefore either $H' \cong \Z$ or $H' = \{I\}$; in the second case, $H$ is finite.
\end{proof}

In case $H$ is finite, it is either trivial or it contains a non-trivial torsion element.
We further analyse the case where $H$ contains a finite index subgroup $H'$ isomorphic to $\Z$.
We need a deep result from Swan:

\begin{lemC}[{\cite[Theorem~B]{SWAN1969585}}]\label{lem:torfreefree}
    A group that is torsion-free and virtually free is free.
\end{lemC}

\begin{lem}\label{lem:tororZ}
    Let $H \leq \SL(2, \Z)$ be a group which contains a finite index subgroup $H'$ isomorphic to $\Z$.
    Then either $H$ contains a non-trivial torsion element, or it is also isomorphic to $\Z$.
\end{lem}
\begin{proof}
    Suppose $H$ is torsion-free.
    Since $H$ is virtually free and torsion-free, by Lemma~\ref{lem:torfreefree}, it is free.
    Hence either $H \cong \Z$, or $H$ is a non-abelian free group.
    But a non-abelian free group is not virtually solvable (Theorem~\ref{thm:Tits}), so it cannot contain a finite index subgroup $H'$ isomorphic to $\Z$.
    Therefore, we must have $H \cong \Z$.
\end{proof}

By Lemma~\ref{lem:solvZ} and \ref{lem:tororZ}, we can already prove the first part of Proposition~\ref{prop:simid}.
In particular, by Lemma~\ref{lem:solvZ}, $H$ is either finite or contains a finite index subgroup isomorphic to $\Z$.
If $H$ is finite, then either it is trivial or it contains a non-trivial torsion element.
If $H$ contains a finite index subgroup isomorphic to $\Z$, then by Lemma~\ref{lem:tororZ}, either it contains a non-trivial torsion element, or it is isomorphic to $\Z$.
Therefore, we have proved that exactly one of the following holds.
    \begin{enumerate}[nosep, label = (\arabic*)]
        \item $H$ is the trivial group.
        \item $H$ contains a non-trivial torsion element.
        \item $H$ is isomorphic to $\Z$.
    \end{enumerate}
In case (3), the generator of $H$ is either a twisted inversion, a shear, an inverting scale, or a positive scale.
This corresponds to the cases (iii)-(vi) of Proposition~\ref{prop:simid}.
We have thus proved the first part of Proposition~\ref{prop:simid}.
The following lemma shows that one can decide which of the six cases in Proposition~\ref{prop:simid} is true.
    
\begin{restatable}{lem}{lemdecisoZ}\label{lem:decisoZ}
    Let $A_1, \ldots, A_K$ be matrices in $\SL(2, \Z)$.
    The following can be done in PTIME:
    \begin{enumerate}[noitemsep, label = (\roman*)]
        \item decide whether the \emph{group} $H \coloneqq \langle A_1, \ldots, A_K \rangle_{grp}$ is trivial.
        \item decide whether $H$ is isomorphic to $\Z$.
        \item compute a generator $A$ of $H$ in case $H \cong \Z$.
    \end{enumerate}
\end{restatable}

\begin{proof}
(i).
    $H$ is trivial if and only if $A_i = I$ for all $i$.
    
(ii).
    First we check whether $H$ is abelian, this is done simply by checking whether $A_i A_j = A_j A_i$ for all $1 \leq i, j \leq K$.
    If $H$ is not abelian, then it is not isomorphic to $\Z$.

    Suppose $H$ is abelian.
    Then the group homomorphism
    \begin{align*}
    \varphi \colon \Z^K & \longrightarrow H \\
    (n_1, \ldots, n_K) & \longmapsto A_1^{n_1} \cdots A_K^{n_K}
    \end{align*}
    is surjective.
    The kernel
    \[
        \Lambda \coloneqq \ker(\varphi) = \{(n_1, \ldots, n_K) \mid A_1^{n_1} \cdots A_K^{n_K} = I\}
    \]
    is a finitely generated subgroup of $\Z^K$.
    A $\Z$-basis of $\Lambda$ is computable in PTIME by a classic result of Babai et al.\ \cite{babai1996multiplicative}.

    Let $\bl_1, \ldots, \bl_m$ be a $\Z$-basis of $\Lambda$.
    Define
    \begin{equation*}
        \overline{\Lambda} \coloneqq \left\{\bn \in \Z^K \;\middle|\; \bn = q_1 \bl_1 + \cdots + q_m \bl_m, \text{ where } q_1, \ldots, q_m \in \Q \right\},
    \end{equation*}
    which is the lattice of integer points in the $\Q$-linear space spanned by $\Lambda$.
    A basis of $\overline{\Lambda}$ can be computed in PTIME the following way.
    One can compute in PTIME a $\Q$-basis of the orthogonal space of $\sum_{i = 1}^m \Q \bl_i$, and thus write $\overline{\Lambda}$ as the integer solution set of a system of homogeneous linear equations of polynomial size.
    One can multiply the coefficients of these homogeneous linear equations by their common denominator, and suppose all the coefficients to be integers.
    From this system, a basis of $\overline{\Lambda}$ is computable in PTIME using row reduction.
    It is then decidable in PTIME whether $\overline{\Lambda} = \Lambda$, by checking whether each element in the basis of $\overline{\Lambda}$ belongs to $\Lambda$.

    We claim that $H$ is torsion-free if and only if $\overline{\Lambda} = \Lambda$.
    Indeed, let $T = A_1^{n_1} \cdots A_K^{n_K}$ be a non-trivial torsion element of $H$, then $T^m = I$ for some $m > 1$.
    Therefore $A_1^{m n_1} \cdots A_K^{m n_K} = I$, so $(m n_1, \ldots, m n_K) \in \Lambda$.
    This shows $(n_1, \ldots, n_K) \in \overline{\Lambda}$.
    But $T \neq I$, so $(n_1, \ldots, n_K) \notin \Lambda$.
    Hence, $T = A_1^{n_1} \cdots A_K^{n_K}$ is a non-trivial torsion element if and only if $(n_1, \ldots, n_K) \in \overline{\Lambda} \setminus \Lambda$.

    This proves the claim. Therefore it is decidable in PTIME whether $H$ contains a non-trivial torsion element.
    By Lemma~\ref{lem:tororZ}, it is decidable in PTIME whether $H \cong \Z$.
    
(iii).
    If $H$ is isomorphic to $\Z$, then the quotient group $\Z^K / \Lambda \cong H$ is isomorphic to $\Z$.
    Using the Hermite Normal Form, one can in PTIME compute an element $\bx = (x_1, \ldots, x_K) \in \Z^K$ such that $\bx, \bl_1, \ldots, \bl_m$ form a $\Z$-basis of $\Z^K$. (Recall that $\bl_1, \ldots, \bl_m$ are a $\Z$-basis of $\Lambda$.)
    Then $\bx + \Lambda$ generates $\Z^K / \Lambda$.
    Therefore $A \coloneqq \varphi(\bx) = A_1^{x_1} \cdots A_K^{x_K}$ generates $H$.
\end{proof}

We now proceed to prove the PTIME decidability claim of Proposition~\ref{prop:simid} for all six cases.

\subsection{$H$ is trivial}
In this case, $\mG = \{(I, \ba_1), \cdots, (I, \ba_K)\}$.

\begin{prop}\label{prop:triv}
    Let $\mG = \{(I, \ba_1), \cdots, (I, \ba_K)\}$.
    Then $\sgmG$ is a group if and only if the equation
    \begin{equation}\label{eq:lin}
        n_1 \ba_1 + n_2 \ba_2 + \cdots n_K \ba_K = \bzer
    \end{equation}
    has a solution $(n_1, \ldots, n_K) \in \Zpp^K$.
    In particular, this is decidable in PTIME.
\end{prop}
\begin{proof}
    Note that the matrices in $\mG$ commute.
    Therefore by Lemma~\ref{lem:grpword}, $\sgmG$ is a group if and only if Equation~\eqref{eq:lin} has a solution $(n_1, \ldots, n_K) \in \Zpp^K$.
    By the homogeneity of Equation~\eqref{eq:lin}, it has a solution in $\Zpp^K$ if and only if it has a solution in $\Q_{>0}^K$.
    This is decidable in PTIME by linear programming~\cite{khachiyan1979polynomial}.
\end{proof}

\subsection{$H$ contains a torsion element}

We show that $\sgmG$ is always a group in case $H$ contains a non-trivial torsion element.

\begin{prop}\label{prop:torid}
    Let $\mG \coloneqq \{(A_1, \ba_1), \ldots, (A_K, \ba_K)\}$ be a set of elements of $\SA$.
    If the semigroup $\langle A_1, \ldots, A_K \rangle$ is a group containing a non-trivial torsion element, then $\sgmG$ is a group.
\end{prop}
\begin{proof}
    Suppose $\langle A_1, \ldots, A_K \rangle$ contains a non-trivial torsion element $T$.
    Let $m > 1$ be such that $T^m = I$.
    Let $\bt \in \Z^2$ be a vector such that $(T, \bt)$ is an element in $\sgmG$ represented by a full-image word (such a word exists by Lemma~\ref{lem:grpword}).
    Then 
    \[
        (T, \bt)^m = (T^m, (I + T + \cdots + T^{m-1})\bt)
        = (I, (I - T)^{-1}(I - T^m)\bt) = (I, \bzer).
    \]
    Here, $I - T$ is invertible because $T \in \SL(2, \Z)$ is non-trivial torsion, so the eigenvalues of $T$ are all different from one ($T$ must be in case (i) or (ii) in the classification of elements in $\SL(2, \Z)$, Section~\ref{sec:prelim}).
    We conclude that $(I, \bzer)$ can be represented by a full-image word over the alphabet $\mG$.
    Hence, $\sgmG$ is a group by Lemma~\ref{lem:grpword}.
\end{proof}

In the next four cases, $H$ is isomorphic to $\Z$.
Let $A$ be a generator of $H$, computable in PTIME by Lemma~\ref{lem:decisoZ}.
The \emph{Jordan Normal Form} of $A$ can be computed in PTIME~\cite{giesbrecht1995nearly}.
One can directly obtain from its Jordan Normal Form whether $A$ is a twisted inversion, a shear or a scale.

\subsection{$H$ is generated by a twisted inversion $A$}
We show that if $A$ is a twisted inversion, then $\sgmG$ is always a group.
\begin{prop}\label{prop:rev}
    If the generator $A$ of the group $\langle A_1, \ldots, A_K \rangle \cong \Z$ is a twisted inversion, then $\sgmG$ is a group.
\end{prop}
\begin{proof}
    Since $\langle A_1, \ldots, A_K \rangle$ is isomorphic to $\Z$, we have $A, A^{-1} \in \langle A_1, \ldots, A_K \rangle$.
    As \linebreak[5] $\langle A_1, \ldots, A_K \rangle$ is a group, let $(A, \ba)$ and $(A^{-1}, \bb)$ be elements of $\sgmG$ represented by full-image words over $\mG$.
    
    We claim that
    \[
    (A, \ba)^2 \cdot (A^{-1}, \bb)^3 \cdot (A, \ba)^2 \cdot (A^{-1}, \bb) = (I, \bzer).
    \]
    By direct computation,
    $
    (A, \ba)^2 \cdot (A^{-1}, \bb)^3 \cdot (A, \ba)^2 \cdot (A^{-1}, \bb) = (I, A^{-1}(A+I)^2 \ba + (A+I)^2 \bb).
    $
    But since $A$ is a twisted inversion, we have $(A+I)^2 = 0$, so the claim is proved.
    We conclude that $(I, \bzer)$ can be represented as a full-image word over $\mG$.
    Hence, $\sgmG$ is a group by Lemma~\ref{lem:grpword}.
\end{proof}

\subsection{$H$ is generated by a shear $A$}
If $A$ is a shear, the main idea in this case is that the semigroup generated by $\mG$ can be embedded as a subsemigroup of the \emph{Heisenberg group}:
\[
\HH_3(\Q) \coloneqq \left\{
\begin{pmatrix}
    1 & a & b \\
    0 & 1 & c \\
    0 & 0 & 1 \\
\end{pmatrix}
\;\middle|\;
a, b, c \in \Q
\right\}.
\]
Since $A$ is a shear, it is triangularizable over $\Q$.
Since $A_1, \ldots, A_K$ are all powers of $A$, they all have eigenvalue one and share the same 1-dimensional invariant subspace, and are therefore simultaneously triangularizable.
Let $P$ be a matrix (with entries in $\Q$) such that $P^{-1} A_i P$ are all of the form 
$
\begin{pmatrix}
        1 & \lambda_i \\
        0 & 1 \\
\end{pmatrix}
$
with $\lambda_i \in \Q$.
Then, we have the following conjugation:
\[
\begin{pmatrix}
        P^{-1} & 0 \\
        0 & 1 \\
\end{pmatrix}
\begin{pmatrix}
        A_i & \ba_i \\
        0 & 1 \\
\end{pmatrix}
\begin{pmatrix}
        P & 0 \\
        0 & 1 \\
\end{pmatrix}
=
\begin{pmatrix}
        P^{-1} A_i P & P^{-1} \ba_i \\
        0 & 1 \\
\end{pmatrix}
\in \HH_3(\Q).
\]

This shows that the map
\begin{equation}\label{eq:defphi}
\varphi : \mG \rightarrow \HH_3(\Q), \quad (A_i, \ba_i) \mapsto 
\begin{pmatrix}
        P^{-1} A_i P & P^{-1} \ba_i \\
        0 & 1 \\
\end{pmatrix},
\end{equation}
extends to an injective semigroup homomorphism from $\sgmG$ to $\HH_3(\Q)$.
Therefore, $\sgmG$ is a group if and only if $\langle \varphi(\mG) \rangle$ is a group.
This is decidable by the following theorem.

\begin{thmC}[\cite{dong2022identity}]\label{thm:H3}
    The Group Problem in $\HH_3(\Q)$ is decidable in PTIME.
\end{thmC}
The elements $\varphi(A_1), \ldots, \varphi(A_K)$ can be computed in PTIME. By Theorem~\ref{thm:H3}, we immediately obtain:
\begin{cor}\label{cor:shear}
    If the generator $A$ of the group $\langle A_1, \ldots, A_K \rangle \cong \Z$ is a shear, then the Group Problem for $\sgmG$ is equivalent to the Group Problem for $\langle \varphi(\mG) \rangle$, which is decidable in PTIME.
\end{cor}

\subsection{$H$ is generated by an inverting scale $A$}
We show that $\sgmG$ is a group in this case.

\begin{restatable}{prop}{proplezero}\label{prop:le0}
    If the generator $A$ of the group $\langle A_1, \ldots, A_K \rangle \cong \Z$ is an inverting scale, then $\sgmG$ is a group.
\end{restatable}
\begin{proof}
    Since $\langle A_1, \ldots, A_K \rangle$ is a group, we have $A, A^{-1} \in \langle A_1, \ldots, A_K \rangle$.
    Recall that $\lambda < -1$ is the smaller eigenvalue of $A$.
    Let $V$ and $W$ be the invariant spaces of $A$ corresponding to the eigenvalues $\lambda$ and $\lambda^{-1}$, and let $\bx_V, \bx_W$ be non-zero vectors respectively in $V$ and $W$.
    In particular we have $A \bx_V = \lambda \bx_V, A \bx_W = \lambda^{-1} \bx_W$.
    As $\langle A_1, \ldots, A_K \rangle$ is a group, let $(A, a_+ \bx_V + b_+ \bx_W)$ and $(A^{-1}, a_- \bx_V + b_- \bx_W)$ be elements of $\sgmG$ represented by full-image words.

    Then for any $m > 0$, the elements
        \begin{align*}
            (I, \bv_m) & \coloneqq (A, a_+ \bx_V + b_+ \bx_W)^m \cdot (A^{-1}, a_- \bx_V + b_- \bx_W)^m \\
            & = \left(A^m, \sum_{i = 0}^{m-1} \lambda^i a_+ \bx_V + \sum_{i = 0}^{m-1} \lambda^{-i} b_+ \bx_W \right) \cdot \left(A^{-m}, \sum_{i = 0}^{m-1} \lambda^{-i} a_- \bx_V + \sum_{i = 0}^{m-1} \lambda^i b_- \bx_W \right) \\
            & = \left(I, \left(\frac{a_+}{1 - \lambda} - \frac{a_-}{1 - \lambda^{-1}}\right)(1 - \lambda^{m})\bx_V + \left(\frac{b_+}{1 - \lambda^{-1}} - \frac{b_-}{1 - \lambda}\right)(1 - \lambda^{-m})\bx_W\right)
        \end{align*}
    and 
        \begin{align*}
            (I, \bw_m) & \coloneqq (A^{-1}, a_- \bx_V + b_- \bx_W)^m \cdot (A, a_+ \bx_V + b_+ \bx_W)^m \\
            & = \left(A^{-m}, \sum_{i = 0}^{m-1} \lambda^{-i} a_- \bx_V + \sum_{i = 0}^{m-1} \lambda^i b_- \bx_W \right) \cdot \left(A^m, \sum_{i = 0}^{m-1} \lambda^i a_+ \bx_V + \sum_{i = 0}^{m-1} \lambda^{-i} b_+ \bx_W \right) \\
            & = \left(I, \left(\frac{a_+}{1 - \lambda} - \frac{a_-}{1 - \lambda^{-1}}\right)(\lambda^{-m} - 1)\bx_V + \left(\frac{b_+}{1 - \lambda^{-1}} - \frac{b_-}{1 - \lambda}\right)(\lambda^{m} - 1)\bx_W\right)
        \end{align*}
    are in $\sgmG$ and represented by full-image words.
    
    Consider four cases:
    \begin{enumerate}
        \item $\frac{a_+}{1 - \lambda} - \frac{a_-}{1 - \lambda^{-1}} = \frac{b_+}{1 - \lambda^{-1}} - \frac{b_-}{1 - \lambda} = 0$.
        In this case we have directly $\bv_m = \bzer$ for all $m$. So $(I, \bzer)$ can be represented by a full-image word.
        \item $\frac{a_+}{1 - \lambda} - \frac{a_-}{1 - \lambda^{-1}} = 0$, but $\frac{b_+}{1 - \lambda^{-1}} - \frac{b_-}{1 - \lambda} \neq 0$.
        When $m$ is even, $\bw_m$ is a positive multiple of $\left(\frac{b_+}{1 - \lambda^{-1}} - \frac{b_-}{1 - \lambda}\right) \bx_W$, and when $n$ is odd, $\bw_n$ is a negative multiple of $\left(\frac{b_+}{1 - \lambda^{-1}} - \frac{b_-}{1 - \lambda}\right) \bx_W$.
        Therefore, there exist positive real numbers $r_1, r_2$ such that $r_1 \bw_m + r_2 \bw_n = \bzer$.
        Since $\bw_m, \bw_n$ have integer entries, there even exist positive \emph{integers} $n_1, n_2$ such that $n_1 \bw_m + n_2 \bw_n = \bzer$.
        Therefore $(I, \bzer) = (I, \bw_m)^{n_1} (I, \bw_n)^{n_2}$ can be represented by a full-image word.

        \item $\frac{b_+}{1 - \lambda^{-1}} - \frac{b_-}{1 - \lambda} = 0$, but $\frac{a_+}{1 - \lambda} - \frac{a_-}{1 - \lambda^{-1}} \neq 0$.
        By symmetry, this case follows from the previous one.

        \item Both $\frac{a_+}{1 - \lambda} - \frac{a_-}{1 - \lambda^{-1}}$ and $\frac{b_+}{1 - \lambda^{-1}} - \frac{b_-}{1 - \lambda}$ are non-zero.
        In this case, consider the vectors $\bv_{2m}$, $\bv_{2m+1}$, $\bw_{2m}$, $\bw_{2m+1}$ when $m$ tends to infinity.
        Since $\lambda < -1$, we have $\lim_{m \rightarrow \infty} (1 - \lambda^{2m}) = - \infty$, $\lim_{m \rightarrow \infty} (1 - \lambda^{2m+1}) = + \infty$, $\lim_{m \rightarrow \infty} (1 - \lambda^{-2m}) = 1$, $\lim_{m \rightarrow \infty} (1 - \lambda^{-2m-1}) = 1$.
        Therefore, the \emph{direction} of $\bv_{2m}$ tends towards $- \left(\frac{a_+}{1 - \lambda} - \frac{a_-}{1 - \lambda^{-1}}\right)\bx_V$, the direction of $\bv_{2m+1}$ tends towards $\left(\frac{a_+}{1 - \lambda} - \frac{a_-}{1 - \lambda^{-1}}\right) \bx_V$, the direction of $\bw_{2m}$ tends towards $\left(\frac{b_+}{1 - \lambda^{-1}} - \frac{b_-}{1 - \lambda}\right) \bx_W$, the direction of $\bw_{2m+1}$ tends towards $- \left(\frac{b_+}{1 - \lambda^{-1}} - \frac{b_-}{1 - \lambda}\right) \bx_W$.
        Hence, when $m$ is large enough, the four vectors $\bv_{2m}$, $\bv_{2m+1}$, $\bw_{2m}$, $\bw_{2m+1}$ in $\Z^2$ generate $\Q^2$ as a $\Qp$-cone.
        Hence, there exist positive integers $t_1, t_2, t_3, t_4$ such that
        \[
        (I, \bv_{2m})^{t_1} (I, \bv_{2m+1})^{t_2} (I, \bw_{2m})^{t_3} (I, \bw_{2m+1})^{t_4} = (I, \bzer).
        \]        
        Therefore $(I, \bzer)$ can be represented by a full-image word.
    \end{enumerate}
    In all cases, $(I, \bzer)$ can be represented by a full-image word, so $\sgmG$ is a group by Lemma~\ref{lem:grpword}.
\end{proof} 

\subsection{$H$ is generated by a positive scale $A$}
This is the most technical case in this section.
We show that if $A$ is a positive scale, then we can decide whether $\sgmG$ is a group in PTIME.
Let $P = (\bx_V, \bx_W) \in \SL(2, \R)$ be a change of basis matrix such that $P^{-1} A P = A'$, where $A'$ is diagonal and can be written as
\[
    A' = 
    \begin{pmatrix}
        \lambda & 0 \\
        0 & \lambda^{-1} \\
    \end{pmatrix}
\]
with $\lambda > 1$.
For $i = 1, \ldots, K$, let $A'_i \coloneqq P^{-1} A_i P$ and $\ba'_i \coloneqq P^{-1} \ba_i$, with
\[
    A'_i = 
    \begin{pmatrix}
        \lambda^{z_i} & 0 \\
        0 & \lambda^{-z_i} \\
    \end{pmatrix}, z_i \in \Z,
    \quad
    \ba'_i = 
    \begin{pmatrix}
        a_i \\
        b_i \\
    \end{pmatrix}.
\]
These are the forms of $A_i$ and $\ba_i$ under the new basis $(\bx_V, \bx_W)$, where $\bx_V, \bx_W$ are eigenvectors of $A$ such that $A \bx_V = \lambda \bx_V, A \bx_W = \lambda^{-1} \bx_W$.

Define the sets
\begin{align*}
J_+ \coloneqq \{i \mid z_i > 0\}, \;
J_- \coloneqq \{i \mid z_i < 0\}, \;
J_0 \coloneqq \{i \mid z_i = 0\}.
\end{align*}
Then $J_+ \cup J_- \cup J_0 = \{1, \ldots, K\}$.
Since $\langle A_1, \ldots, A_K \rangle$ is isomorphic to $\Z$, we have $J_+ \neq \emptyset, J_- \neq \emptyset$.

Divide the set $\R^2$ into nine parts:
\begin{align*}
    & \mR_{++} = \{(x, y) \mid x > 0, y > 0\}, 
     \mR_{+0} = \{(x, y) \mid x > 0, y = 0\}, 
     \mR_{+-} = \{(x, y) \mid x > 0, y < 0\}, \\
    & \mR_{0+} = \{(x, y) \mid x = 0, y > 0\}, 
     \mR_{00} = \{(x, y) \mid x = 0, y = 0\}, 
     \mR_{0-} = \{(x, y) \mid x = 0, y < 0\}, \\
    & \mR_{-+} = \{(x, y) \mid x < 0, y > 0\}, 
     \mR_{-0} = \{(x, y) \mid x < 0, y = 0\}, 
     \mR_{--} = \{(x, y) \mid x < 0, y < 0\}.
\end{align*}
Each part is called a \emph{cell}. See Figure~\ref{fig:cellsSA}.
    \begin{figure}[ht!]
        \centering
        \begin{minipage}[t]{.45\textwidth}
            \centering
            \includegraphics[width=\textwidth, keepaspectratio, trim={3cm 0cm 2cm 0cm},clip]{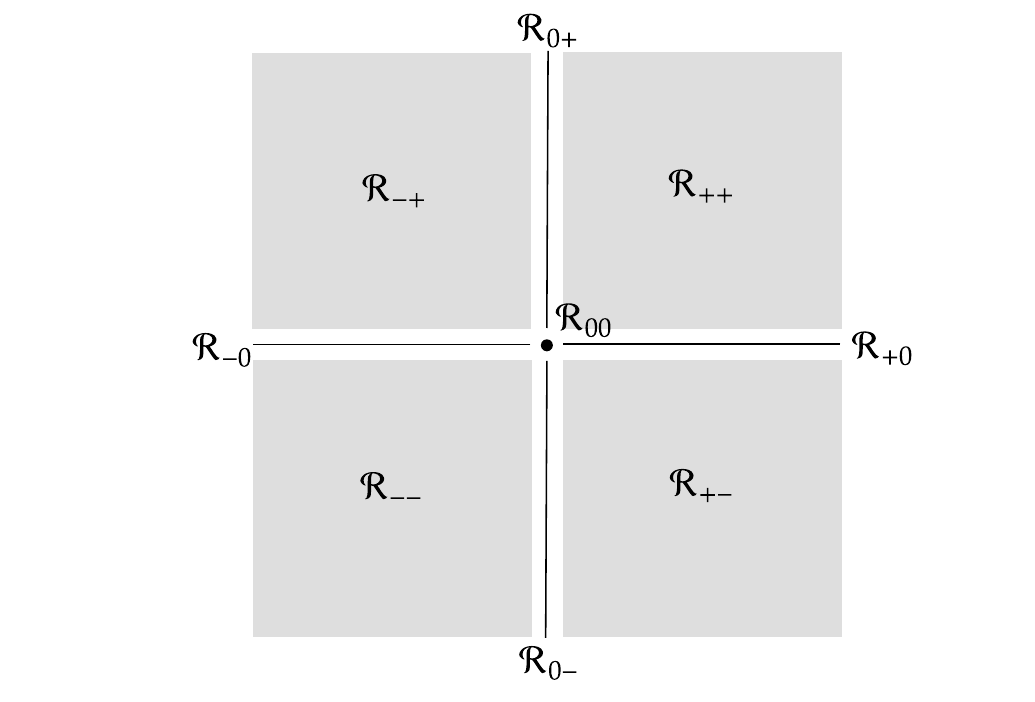}
            \caption{Cells.}
            \label{fig:cellsSA}
        \end{minipage}
        \hfill
        \begin{minipage}[t]{0.45\textwidth}
            \centering
            \includegraphics[width=0.9\textwidth, keepaspectratio, trim={4cm 0.5cm 3cm 0cm},clip]{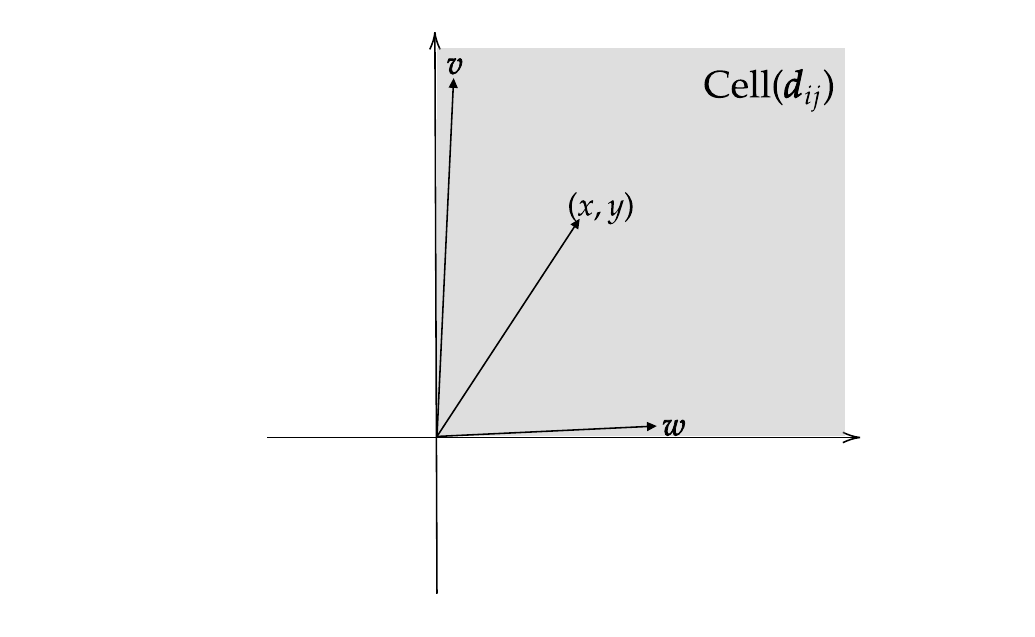}
            \caption{Illustration for Lemma~\ref{lem:fullcell}.}
            \label{fig:vw}
        \end{minipage}
    \end{figure}
    
For an element $\bx \in \R^2$, denote by $\Cell(\bx)$ the cell which it belongs to.
Writing $\bx = (x_1, x_2)^{\top}$, it is easy to see that
\begin{equation}\label{eq:defcell}
\Cell(\bx) = \left\{(r_1 x_1, r_2 x_2)^{\top} \;\middle|\; r_1, r_2 \in \Rpp \right\}.
\end{equation}
In particular, $\bx$ and ${A'}^{z} \bx$ are in the same cell for all integers $z$.

For each $i \in J_+ \cup J_- \cup J_0$, denote 
\begin{align*}
n_i \coloneqq
\begin{cases}
|z_i| \quad & z_i \neq 0 \\
1 \quad & z_i = 0.
\end{cases}
\end{align*}
For each tuple $(i, j) \in J_+ \times J_-$, define the vector
\[
    \bd_{ij} \coloneqq ( d_{ija}, d_{ijb} )^{\top} \in \R^2
\]
where
\[
d_{ija} \coloneqq - \frac{a_i}{1 - \lambda^{n_i}} + \frac{a_j}{1 - \lambda^{-n_j}}, \;\; d_{ijb} \coloneqq \frac{b_i}{1 - \lambda^{-n_i}} - \frac{b_j}{1 - \lambda^{n_j}}.
\]
These expressions are defined so that
\[
({A'_i}, \ba'_i)^{n_j} ({A'_j}, \ba'_j)^{n_i} = \left(I, (1 - \lambda^{- n_i n_j}) \cdot \left( \lambda^{n_i n_j} d_{ija}, d_{ijb} \right)^{\top}\right)
\]
for all $(i, j) \in J_+ \times J_-$.

For each element $k \in J_0$, define the vector
\[
    \be_{k} \coloneqq \left(a_k, b_k \right)^{\top} \in \R^2.
\]
Similarly, this expression is defined so that
\[
({A'_k}, \ba'_k)^{n_k} = \left(I, \be_k \right)
\]
for all $k \in J_0$.
Denote by $\mG'$ the new alphabet $\{(A'_1, \ba'_1), \ldots, (A'_K, \ba'_K)\}$.
The following lemma shows that the rays $\bv \Rpp$ such that $(I, \bv) \in \langle \mG' \rangle$ can reach both ``extremities'' of the cells $\Cell(\bd_{ij}), (i, j) \in J_+ \times J_-$.
See Figure~\ref{fig:vw} for an illustration.

\begin{lem}\label{lem:fullcell}
    Let $(i, j)$ be a pair in $J_+ \times J_-$ and $(x, y)^{\top} \in \R^2$ be a vector in $\Cell(\bd_{ij})$.
    If $\lambda > 1$, then there exist elements $(I, \bv)$ and $(I, \bw)$ in $\langle\mG'\rangle$, where the vector $(x, y)^{\top}$ can be written as $r_1 \bv + r_2 \bw$ for some $r_1, r_2 \in \Rpp$.
    Furthermore, $(I, \bv)$ and $(I, \bw)$ are represented by words over the alphabet $\mG'$, such that the letters $(A'_i, \ba'_i)$ and $(A'_j, \ba'_j)$ both occur.
\end{lem}
\begin{proof}
    For any $p \in \Zpp$, let $\bv_p, \bw_p$ be such that
    \[
        (I, \bv_p)  \coloneqq ({A'_i}, \ba'_i)^{p n_j} ({A'_j}, \ba'_j)^{p n_i}, \quad
        (I, \bw_p)  \coloneqq ({A'_j}, \ba'_j)^{p n_i} ({A'_i}, \ba'_i)^{p n_j}.
    \]
    By direct computation, we have
    \begin{equation}\label{eq:justify}
        \bv_p  = (1 - \lambda^{- p n_i n_j}) \cdot \left( \lambda^{p n_i n_j} d_{ija}, d_{ijb} \right)^{\top}, \quad
        \bw_p  = (1 - \lambda^{- p n_i n_j}) \cdot \left( d_{ija}, \lambda^{p n_i n_j} d_{ijb} \right)^{\top}.
    \end{equation}
    Since $\lambda > 1$, we have $1 - \lambda^{- p n_i n_j} > 0$.

    If $\Cell(\bd_{ij})$ is of dimension one or zero, then $d_{ija} = 0$ or $d_{ijb} = 0$.
    In both cases, $(x, y)^{\top}, \bv_p, \bw_p$ are linearly dependent and have the same direction, so $(x, y)^{\top}$ can be written as $r_1 \bv_p + r_2 \bw_p$ for some $r_1, r_2 \in \Rpp$.

    If $\Cell(\bd_{ij})$ is of dimension two, then $d_{ija} \neq 0$ and $d_{ijb} \neq 0$.
    Let $p$ be large enough so that $\frac{x}{y}$ is between $\frac{\lambda^{pn_i n_j} d_{ija}}{d_{ijb}}$ and $\frac{d_{ija}}{\lambda^{pn_i n_j} d_{ijb}}$.
    Then $(x, y)^{\top}$ can be written as $r_1 \bv_p + r_2 \bw_p$ for some $r_1, r_2 \in \Rpp$.
\end{proof} 

A similar lemma can be shown for $\Cell(\be_{k})$: the rays $\bv \Rpp$ such that $(I, \bv) \in \langle \mG' \rangle$ can fill up the cells $\Cell(\be_{k}), k \in J_0$.

\begin{restatable}{lem}{lemfullcelle}\label{lem:fullcell2}
    Suppose $J_+$ and $J_-$ are non-empty.
    Let $k \in J_0$, and $(x, y)^{\top}  \in \Cell(\be_{k})$.
    If $\lambda > 1$, then there exist elements $(I, \bv)$ and $(I, \bw)$ in $\langle\mG'\rangle$, such that the vector $(x, y)^{\top}$ can be written as $r_1 \bv + r_2 \bw$ for some $r_1, r_2 \in \Rpp$.
    Furthermore, $(I, \bv)$ and $(I, \bw)$ are represented as words over the alphabet $\mG'$, such that the letter $(A'_k, \ba'_k)$ occurs.
\end{restatable}
\begin{proof}
    Let $(i, j)$ be any pair in $J_+ \times J_-$.
    For any $p \in \Zp$, $q \in \Zpp$, denote
    \[
        (I, \bv_{pq}) \coloneqq {A'_i}^{p n_j} {A'_k}^{q} {A'_j}^{p n_i}, \quad (I, \bw_{pq}) \coloneqq {A'_j}^{p n_i} {A'_k}^{q} {A'_i}^{p n_j},
    \]
    By direct computation, we have
    \[
        \bv_{pq} = (1 - \lambda^{- p n_i n_j}) \cdot \left( \lambda^{p n_i n_j} d_{ija}, d_{ijb} \right)^{\top} 
        + q \cdot (\lambda^{p n_i n_j} a_k, \lambda^{- p n_i n_j} b_k)^{\top},
    \]
    \[
        \bw_{pq} = (1 - \lambda^{- p n_i n_j}) \cdot \left( d_{ija}, \lambda^{p n_i n_j} d_{ijb} \right) ^{\top}
        + q \cdot  (\lambda^{- p n_i n_j} a_k, \lambda^{p n_i n_j} b_k)^{\top}.
    \]
    Since $\lambda > 1$, we have $1 - \lambda^{- p n_i n_j} > 0$.

    If $\Cell(\be_{k})$ has dimension one or zero, then take $p = 0$ and the statement is trivial.
    Suppose that $\Cell(\be_{k})$ has dimension two, then we have $a_k \neq 0$ and $b_k \neq 0$.
    Fix a large enough $p$ so that $\frac{x}{y}$ is between $\frac{\lambda^{pn_i n_j} a_k}{\lambda^{- pn_i n_j} b_k}$ and $\frac{\lambda^{-pn_i n_j} a_k}{\lambda^{pn_i n_j} b_k}$.
    Then $(x, y)^{\top}$ can be written as $r_1 (\lambda^{p n_i n_j} a_k, \lambda^{- p n_i n_j} b_k)^{\top} + r_2 (\lambda^{- p n_i n_j} a_k, \lambda^{p n_i n_j} b_k)^{\top}$ for some $r_1, r_2 \in \Rpp$.
    When $q$ tends towards infinity, the direction of the vectors $\bv_{pq}, \bw_{pq}$ tends respectively to $(\lambda^{p n_i n_j} a_k, \lambda^{- p n_i n_j} b_k)^{\top}$ and $(\lambda^{- p n_i n_j} a_k, \lambda^{p n_i n_j} b_k)^{\top}$.
    Therefore, for large enough $q$, the vector $(x, y)^{\top}$ can be written as $r'_1 \bv_{pq} + r'_2 \bw_{pq}$ for some $r'_1, r'_2 \in \Rpp$.
\end{proof} 

Define the \emph{radical} $\widehat{(A'_i, \ba'_i)}$ of $(A'_i, \ba'_i)$ as
    \[
        \widehat{(A'_i, \ba'_i)} \coloneqq \left(
        \begin{pmatrix}
            \lambda & 0 \\
            0 & \lambda^{-1} \\
        \end{pmatrix},
        \begin{pmatrix}
            a_i \cdot \frac{1 - \lambda}{1 - \lambda^{n_i}} \\
            b_i \cdot \frac{1 - \lambda^{-1}}{1 - \lambda^{-{n_i}}} \\
        \end{pmatrix}
        \right)
    \]
    if $z_i = n_i > 0$, and
    \[
        \widehat{(A'_i, \ba'_i)} \coloneqq \left(
        \begin{pmatrix}
            \lambda^{-1} & 0 \\
            0 & \lambda \\
        \end{pmatrix},
        \begin{pmatrix}
            a_i \cdot \frac{1 - \lambda^{-1}}{1 - \lambda^{- n_i}} \\
            b_i \cdot \frac{1 - \lambda}{1 - \lambda^{n_i}} \\
        \end{pmatrix}
        \right)
    \]
    if $z_i = -n_i < 0$, and 
    \[
        \widehat{(A'_i, \ba'_i)} \coloneqq (A'_i, \ba'_i)
    \]
    if $z_i = 0$.
    This is defined so that $\widehat{(A'_i, \ba'_i)}^{n_i} = (A'_i, \ba'_i)$ in all cases.

Define the alphabet
    \[
        \widehat{\mG} \coloneqq \left\{\widehat{(A'_1, \ba'_1)}, \ldots, \widehat{(A'_K, \ba'_K)}\right\}.
    \]
Define the following union of cells:
    \begin{equation}\label{eq:defS}
        \mS \coloneqq \left(\bigcup_{i \in J_+, j \in J_-} \Cell(\bd_{ij}) \right) \bigcup \left( \bigcup_{k \in J_0} \Cell(\be_{k}) \right).
    \end{equation}

\begin{restatable}{lem}{lemradcomb}\label{lem:radcomb}
    Let $w \coloneqq (C_1, \bc_1) \cdots (C_M, \bc_M)$ be a full-image word over the alphabet $\widehat{\mG}$, such that $(C_1, \bc_1) \cdot \cdots (C_M, \bc_M) = (I, \bx)$.
    Then there exists a finite non-empty set of vectors $\{\bs_1, \ldots, \bs_m\} \subseteq \mS$ such that the following conditions are satisfied:
    \begin{enumerate}[noitemsep, label = (\roman*)]
        \item $r_1 \bs_1 + \cdots + r_m \bs_m = \bx$ for some strictly positive reals $r_1, \ldots, r_m$.
        \item For each $i \in J_+$, there exist $j \in J_-$ and $\ell \in \{1, \ldots, m\}$, such that $\bs_{\ell} \in \Cell(\bd_{ij})$.
        \item For each $j \in J_-$, there exist $i \in J_+$ and $\ell \in \{1, \ldots, m\}$, such that $\bs_{\ell} \in \Cell(\bd_{ij})$.
        \item For each $k \in J_0$, there exist $\ell \in \{1, \ldots, m\}$ such that $\bs_{\ell} \in \Cell(\be_{k})$.
    \end{enumerate}
\end{restatable}
\begin{proof}
    We call a vector of the form $r_1 \bs_1 + \cdots + r_m \bs_m, m \geq 1, r_i \in \Rpp, \bs_i \in \mS$, an \emph{$\Rpp$-linear combination} of elements in $\mS$.
    For $i = 1, \ldots, M$, write
    \[
    (C_i, \bc_i) = \left(
        \begin{pmatrix}
            \lambda^{t_i} & 0 \\
            0 & \lambda^{-t_i} \\
        \end{pmatrix},
        \begin{pmatrix}
            c_i \\
            d_i \\
        \end{pmatrix}
        \right),
    \]
    where $t_i \in \{-1, 0, 1\}$.
    We show that if $(C_1, \bc_1) \cdot \cdots (C_M, \bc_M) = (I, \bx)$ then $\bx$ can be written as an $\Rpp$-linear combination of elements in $\mS$.
    We use induction on $M$.
    When $M = 1$, $(C_1, \bc_1) = (I, \be_k)$ for some $k \in J_0$, and the statement is obvious.
    When $M \geq 2$, distinguish the following three cases.
    
    \begin{enumerate}[label = (\arabic*)]
    
        \item \textbf{$t_1 t_M = 1$.}
        Suppose $t_1 = t_M = 1$, the case where $t_1 = t_M = -1$ can be done analogously.

        Since $t_1 = 1 > 0$ and $t_1 + \cdots + t_{M-1} = -1 < 0$, there must exist $2 \leq i \leq M-2$ such that $t_1 + \cdots + t_i = 0$.
        By induction hypothesis, 
        \[
            (C_1, \bc_1) \cdots (C_{i}, \bc_{i}) = (I, \byy), \quad (C_{i+1}, \bc_{i+1}) \cdots (C_{M}, \bc_{M}) = (I, \byy'),
        \]
        where $\byy$ and $\byy'$ can be written as an $\Rpp$-linear combination of elements in $\mS$.
        Therefore $\bx = \byy + \byy'$ also satisfies this claim.
        
        \item \textbf{$t_1 t_M = - 1$.}
        In this case, $C_2 \cdots C_{M-1} = C_1 C_2 \cdots C_{M}$ $= I$.
        Define 
        \[
        w' \coloneqq (C_2, \bc_2) \cdots (C_{M-1}, \bc_{M-1}).
        \]
        Then the product of $w'$ is of the form $(I, \bx')$, where $\bx'$ is either $\bzer$ (when $w'$ is empty), or $\bx'$ is an $\Rpp$-combination of elements in $\mS$ by the induction hypothesis.
        Hence
        \[
            (C_1, \bc_1) \cdots (C_{M}, \bc_{M})
            = (C_1, \bc_1) \cdot (I, \bx') \cdot (C_{M}, \bc_{M})
            = (I, \bc_1 + C_1 \bc_{M} + C_1 \bx').
        \]
        We claim that $\bc_1 + C_1 \bc_{M} \in \Cell(\bd_{ij})$ for some indices $i \in J_+, j \in J_-$.
        First suppose $t_1 = 1, t_M = - 1$.
        Let $i \in J_+, j \in J_-$ be indices such that $\widehat{(A'_i, \ba'_i)} = (C_1, \bc_1)$ and $\widehat{(A'_j, \ba'_j)} = (C_M, \bc_M)$, then
        \begin{align*}
        \bc_1 + C_1 \bc_{M} = & \left(a_j \cdot \frac{\lambda - 1}{1 - \lambda^{- n_j}} + a_i \cdot \frac{1 - \lambda}{1 - \lambda^{n_i}}, b_j \cdot \frac{\lambda^{-1} - 1}{1 - \lambda^{n_j}} + b_i \cdot \frac{1 - \lambda^{-1}}{1 - \lambda^{-n_i}} \right)^{\top} \\
        = & \left( (\lambda - 1) d_{ija}, (1 - \lambda^{-1}) d_{ijb} \right)^{\top} \\
        \in & \Cell(\bd_{ij}). \quad \quad \quad \quad \text{ (by Equation~\eqref{eq:defcell}) }
        \end{align*}

        Next suppose $t_1 = -1, t_M = 1$.
        Let $i \in J_+, j \in J_-$ be indices such that $\widehat{(A'_j, \ba'_j)} = (C_1, \bc_1)$ and $\widehat{(A'_i, \ba'_i)} = (C_M, \bc_M)$, then
        \begin{align*}
        \bc_1 + C_1 \bc_{M} = & \left(a_i \cdot \frac{\lambda^{-1} - 1}{1 - \lambda^{n_i}} + a_j \cdot \frac{1 - \lambda^{-1}}{1 - \lambda^{- n_j}}, b_i \cdot \frac{\lambda - 1}{1 - \lambda^{- n_i}} + b_j \cdot \frac{1 - \lambda}{1 - \lambda^{n_j}} \right)^{\top} \\
        = & \left( (1 - \lambda^{-1}) d_{ija}, (\lambda - 1) d_{ijb} \right)^{\top} \\
        \in & \Cell(\bd_{ij}).
        \end{align*}

        Hence, in both cases, $\bc_1 + C_1 \bc_{M} \in \Cell(\bd_{ij}) \subseteq \mS$.
        We then show that $\bc_1 + C_1 \bc_{M} + C_1 \bx'$ is an $\Rpp$-combination of elements in $\mS$.
        Since $\bx'$ is either zero or an $\Rpp$-combination of elements in $\mS$, write $\bx' = \sum_{i=1}^m r_i \bs_i$, where $m \geq 0, r_i > 0, \bs_i \in \mS$.
        Then $C_1 \bx' = \sum_{i=1}^m r_i C_1 \bs_i$ is still an $\Rpp$-combination of elements in $\mS$ by Equation~\eqref{eq:defcell}.
        Hence, $\bc_1 + C_1 \bc_{M} + C_1 \bx'$ is an $\Rpp$-combination of elements in $\mS$.
        \item \textbf{$t_1 t_M = 0$.}
        Suppose $t_1 = 0$, the case where $t_M = 0$ can be done analogously.
        
        By induction hypothesis, 
        \[
            (C_1, \bc_1) = (I, \byy), \quad (C_{2}, \bc_{2}) \cdots (C_{M}, \bc_{M}) = (I, \byy'),
        \]
        where $\byy$ and $\byy'$ can be written as an $\Rpp$-linear combination of elements in $\mS$.
        Therefore $\bx = \byy + \byy'$ also satisfies this claim.
    \end{enumerate}
    Therefore we have found an $\Rpp$-linear combination that satisfies (i).
    The following can be easily seen from the above induction procedure: if for some $i \in J_+$ the letter $\widehat{(A'_i, \ba'_i)}$ appears in $w$, then the $\Rpp$-linear combination contains a vector $\bs_{\ell}$ in the cell $\Cell(\bd_{ij})$ for some $j \in J_-$.
    Since $w$ is full-image, the condition (ii) in the statement of the lemma must hold.
    Similarly, the conditions (iii) and (iv) must also hold.
\end{proof}

\begin{prop}\label{prop:idtocells}
    The semigroup $\sgmG$ is a group if and only if there exists a finite set of vectors $\{\bs_1, \ldots, \bs_m\} \subseteq \mS$ that satisfies the four conditions (i)-(iv) in Lemma~\ref{lem:radcomb} with $\bx = \bzer$.
\end{prop}
\begin{proof}
    First suppose there exists a finite set of vectors $\{\bs_1, \ldots, \bs_m\} \subseteq \mS$ satisfying (i)-(iv), we want to find a full-image word $w$ over the alphabet $\mG$, representing $(I, \bzer)$.

    For any $t \in \{1, \ldots, m\}$, if $\bs_t$ is an element of $\Cell(\bd_{ij})$, then by Lemma~\ref{lem:fullcell}, there exist $(I, \bv), (I, \bw) \in \langle \mG' \rangle$ such that $\bs_t = r_1 \bv + r_2 \bw$ for some $r_1, r_2 > 0$.
    Also, the letters $(A'_i, \ba'_i)$ and $(A'_j, \ba'_j)$ appear in some words representing $(I, \bv)$ and $(I, \bw)$.
    If $\bs_t$ is an element of $\Cell(\be_{k})$, then by Lemma~\ref{lem:fullcell2}, there exist $(I, \bv), (I, \bw) \in \langle \mG' \rangle$ such that $\bs_t = r_1 \bv + r_2 \bw$ for some $r_1, r_2 > 0$.
    Also, the letter $(A'_k, \ba'_k)$ appears in some words representing $(I, \bv)$ and $(I, \bw)$.
    
    Therefore, $\bs_t$ can always be written as a strictly positive linear combination of vectors $\bv$ with $(I, \bv) \in \langle \mG' \rangle$.
    Hence, by condition (i), there exist strictly positive reals $r_1, \ldots, r_M$ and vectors $\bv'_1, \ldots, \bv'_M$, such that $r_1 \bv'_1 + \cdots + r_M \bv'_M = \bzer$ and $(I, \bv'_t) \in \langle \mG' \rangle$ for $t = 1, \ldots, M$.
    Furthermore, conditions (ii), (iii) and (iv) show that every letter $(A'_i, \ba'_i)$ in $\mG'$ appears at least once in a word representing $(I, \bv'_t)$ for some $t \in \{1, \ldots, M\}$.

    Changing back to the original basis, this shows that $r_1 \bv_1 + \cdots + r_M \bv_M = \bzer$, where $(I, \bv_t) \in \langle \mG \rangle$ for $t = 1, \ldots, M$.
    Since the entries of $\bv_t$ are all integers, there exist strictly positive \emph{integers} $n_1, \ldots, n_M$, such that $n_1 \bv_1 + \cdots + n_M \bv_M = \bzer$.
    Hence
    \[
        (I, \bv_1)^{n_1} \cdots (I, \bv_M)^{n_M} = (I, \bzer).
    \]
    Every letter $(A_i, \ba_i)$ in $\mG$ appears at least once in a word representing $(I, \bv_t)$ for some $t \in \{1, \ldots, M\}$.
    Therefore, $(I, \bzer)$ can be represented as a full-image word.
    This shows that $\sgmG$ is a group by Lemma~\ref{lem:grpword}.

    Next, suppose $\sgmG$ is a group. We show that there exists an $\Rpp$-linear combination of elements in $\mS$ equal to $\bzer$, that satisfies the conditions (ii), (iii) and (iv).
    
    By Lemma~\ref{lem:grpword}, there exists a full-image word $w \coloneqq (B'_1, \bb'_1) \cdots (B'_m, \bb'_m)$ over the alphabet $\mG'$, representing $(I, \bzer)$.
    Replacing each letter $(B'_i, \bb'_i)$ in $w$ with $n_i$ consecutive letters $\widehat{(B'_i, \bb'_i)}$, we obtain a full-image word 
    \[
        \widehat{w} = (C_1, \bc_1) \cdots (C_M, \bc_M)
    \]
    over the alphabet $\widehat{\mG}$, representing $(I, \bzer)$.

Then by Lemma~\ref{lem:radcomb}, the vector $\bzer$ can be written as an $\Rpp$-linear combination of a finite number of elements in $\mS$, satisfying the conditions (ii), (iii) and (iv).
\end{proof}

\begin{cor}\label{cor:simdec}
    If the generator $A$ of the group $\langle A_1, \ldots, A_K \rangle \cong \Z$ is a positive scale, it is decidable in PTIME whether $\sgmG$ is a group.
\end{cor}
\begin{proof}
    Since $\langle A_1, \ldots, A_K \rangle \cong \Z$, we have $J_+ \neq \emptyset, J_- \neq \emptyset$.
    Given $\mS$ as a set of cells, it is decidable whether there exist a finite non-empty set of vectors $\{\bs_1, \ldots, \bs_m\} \subseteq \mS$ satisfying conditions (ii), (iii) and (iv), as well as strictly positive reals $r_1, \ldots, r_m$ such that $r_1 \bs_1 + \cdots + r_m \bs_m = \bzer$.
    Indeed, this is true if and only if there exists a linear subspace $\mL$ contained in the $\Rp$-cone generated by $\mS$, with the property that $\mL$ intersects some cell $\Cell(\bd_{i*})$ for each $i \in J_+$, some cell $\Cell(\bd_{*j})$ for each $j \in J_-$, and some cell $\Cell(\be_{k})$ for each $k \in J_0$.
    This is decidable in PTIME since the number of cells in $\mL$ is finite (at most nine).
    
    Given the input set $\mG$, we can compute the set of cells in $\mS$ in PTIME.
    Therefore, by Proposition~\ref{prop:idtocells}, it is decidable in PTIME whether $\sgmG$ is a group.
\end{proof}

Combining all cases, we can now prove the main result of this section.

\propsimid*
\begin{proof}
Division into six cases has already been proved by Lemma~\ref{lem:solvZ} and \ref{lem:tororZ} as well as the discussion that follows.
In cases (ii), (iii) and (v), Proposition~\ref{prop:torid}, \ref{prop:rev} and \ref{prop:le0} show that $\sgmG$ is a group.
We now show PTIME decidability.

First, we decide in PTIME which of the six cases is true for $H$, using Lemma~\ref{lem:decisoZ}.
In cases (ii), (iii) and (v), the Group Problem for $\sgmG$ has a positive answer.
In cases (i), (iv) and (vi), Proposition~\ref{prop:triv}, Corollary~\ref{cor:shear} and Corollary~\ref{cor:simdec} show the required PTIME decidability result.
\end{proof}

\section{Extensions and obstacles to Semigroup Membership}\label{sec:obs}
In previous sections we showed decidability and NP-completeness of the Identity Problem and the Group Problem in $\SA$.
In this section we discuss possible extensions of our result and obstacles to solving the Semigroup Membership problem in $\SA$.

Let $\mG \coloneqq \{(A_1, \ba_1), \ldots, (A_K, \ba_K)\}$ be a set of elements of $\SA$.
This first obvious obstacle for deciding Semigroup Membership for $\sgmG$ is that we can no longer suppose $G \coloneqq \langle A_1, \ldots, A_K \rangle$ to be a group, which we could do for the Identity Problem and the Group Problem.
However, if we restrict the target element to be of the form $(I, \ba)$, that is, if we want to deciding whether $(I, \ba) \in \sgmG$, then we can still suppose $G$ to be a group.

Indeed, if $G$ is not a group, then at least one of the $A_i$ is not invertible in $G$. Therefore, a word over $\mG$ representing $(I, \ba)$ cannot contain the letter $(A_i, \ba_i)$.
We can thus delete $(A_i, \ba_i)$ from the alphabet $\mG$ without changing whether $(I, \ba) \in \sgmG$.
One can repeat this process until $G$ becomes a group.

Under the additional assumption that $G$ is a group, we can decide Semigroup Membership for $\sgmG$ in all cases except one.
Recall Theorem~\ref{thm:Tits}. If $G$ contains a non-abelian free subgroup, then $\sgmG$ is a group by Proposition~\ref{prop:nonsimid}.
Hence Semigroup Membership reduces to Group Membership, and is decidable by the result of Delgado~\cite{delgado2017extensions}.
If $G$ is virtually solvable, then consider the six cases in Proposition~\ref{prop:simid}.
Case (ii), (iii) and (v) are easy since $\sgmG$ becomes a group.
In case (i), deciding whether $(I, \ba) \in \sgmG$ reduces to solving the linear equation $n_1 \ba_1 + n_2 \ba_2 + \cdots n_K \ba_K = \ba$ for $(n_1, \ldots, n_K) \in \Zp^K \setminus \{\bzer\}$, and is decidable by integer programming.
In case (iv), Semigroup Membership for $\sgmG$ reduces to Semigroup Membership in the Heisenberg group $\HH_3(\Q)$, which is decidable by the result of Colcombet et al.~\cite{colcombet2019reachability}.
Only case (vi) remains unsolved.

We now show that solving the Semigroup Membership problem in case (vi) is equivalent to solving the Semigroup Membership problem in the group $\Z[\lambda] \rtimes_{\lambda} \Z$.
Given $\lambda > 1$ that satisfies a quadratic equation $\lambda^2 - a \lambda + 1 = 0$ for some $a \geq 3$, we define the following semidirect product:
\[
\Z[\lambda] \rtimes_{\lambda} \Z \coloneqq \left\{
\begin{pmatrix}
    \lambda^k & x \\
    0 & 1\\
\end{pmatrix}
\;\middle|\;
k \in \Z, x \in \Z[\lambda]
\right\}.
\]
Here, $\Z[\lambda]$ is the ring generated by $1$ and $\lambda$.
There exist an embedding of $\Z[\lambda] \rtimes_{\lambda} \Z$ as a subgroup of $\SA$ in the following way.
For a given $\lambda$ satisfying $\lambda^2 - a \lambda + 1 = 0$, define the matrix 
$
A_{\lambda} \coloneqq
\begin{pmatrix}
    0 & -1 \\
    1 & a \\
\end{pmatrix}
$
in $\SL(2, \Z)$.
Since $\lambda$ is a quadratic integer, every element $x \in \Z[\lambda]$ can be written uniquely as $x = x_1 + \lambda x_{\lambda}$ for some $x_1, x_{\lambda} \in \Z$.
Define the map
\begin{align*}
    \phi \colon \Z[\lambda] \rtimes_{\lambda} \Z & \longrightarrow \SA \\
    \begin{pmatrix}
    \lambda^k & x \\
    0 & 1\\
    \end{pmatrix} 
    & \longmapsto \left(A_{\lambda}^k, (x_1, x_{\lambda})^{\top}\right).
\end{align*}
It is easy to verify that $\phi$ is an injective group homomorphism.
Furthermore, the image under $\phi$ of a sub-semigroup\footnote{We suppose that the upper-left entries of elements of the semigroup contain both positive and negative exponents of $\lambda$, otherwise deciding Semigroup Membership is easy.} of $\Z[\lambda] \rtimes_{\lambda} \Z$ satisfies case (vi) of Proposition~\ref{prop:simid}.
Therefore, solving the hard case of the Semigroup Membership problem in $\SA$ necessitates solving the Semigroup Membership problem in $\Z[\lambda] \rtimes_{\lambda} \Z$.

On the other hand, given any positive scale $A \in \SL(2, \Z)$ with eigenvalue $\lambda > 1$, one can find a change of basis matrix $P \in \SL(2, \Z)$ such that $P^{-1}AP = A_{\lambda}$ (see~\cite{CAMPBELL1991175}).
Therefore, any (finitely generated) semigroup $\sgmG$ satisfying case (vi) of Proposition~\ref{prop:simid} must be conjugate\footnote{The conjugation is realized by the change-of-basis matrix $\diag(P, 1).$} to a (finitely generated) sub-semigroup of $\phi(\Z[\lambda] \rtimes_{\lambda} \Z) \leq \SA$.
Hence, solving Semigroup Membership in $\Z[\lambda] \rtimes_{\lambda} \Z$ is sufficient for solving Semigroup Membership in the case (vi) of Proposition~\ref{prop:simid}.

Although the Semigroup Membership problem in $\Z[\lambda] \rtimes_{\lambda} \Z$ remains an open problem, the group bears certain similarities to the better studied \emph{Baumslag-Solitar group} $\mathsf{BS}(1, q)$:
\[
\mathsf{BS}(1, q) \coloneqq 
\Z[\frac{1}{q}] \rtimes_{q} \Z = \left\{
\begin{pmatrix}
    q^k & x \\
    0 & 1\\
\end{pmatrix}
\;\middle|\;
k \in \Z, x \in \Z[\frac{1}{q}]
\right\}.
\]
Here, $q \geq 2$ is an integer.
A recent result by Cadilhac, Chistikov and Zetzsche~\cite{DBLP:conf/icalp/CadilhacCZ20} showed decidability of the \emph{rational subset membership problem} in $\mathsf{BS}(1, q)$ by considering rational languages of \emph{base-$q$ expansions}.
This result subsumes decidability of the Semigroup Membership problem in $\mathsf{BS}(1, q)$.
Therefore, it would be interesting to adapt this approach to study Semigroup Membership in $\Z[\lambda] \rtimes_{\lambda} \Z$ by considering rational languages of \emph{base-$\lambda$ expansions}~\cite{BLANCHARD1989131}, where $\lambda$ is an algebraic integer.
Nevertheless, adaptation of the previous result to a non-integer base of numeration poses various additional difficulties that we have not been able to surmount.

\bibliographystyle{alphaurl}
\bibliography{aff2}

\end{document}